\documentclass[11pt]{amsart}
\usepackage{latexsym,amssymb,amsmath,hyperref,amsthm,amsfonts,
caption,subcaption,tikz,comment}
\usepackage{mdwlist,dirtytalk,times,enumitem}
\usepackage[capitalise, noabbrev]{cleveref}
\usetikzlibrary{
intersections, arrows.meta,
automata,er,calc,
backgrounds,
mindmap,folding,
patterns,
decorations.markings,
fit,decorations,
shapes,matrix,
positioning,
shapes.geometric,
arrows,through, graphs, graphs.standard
}
\usepackage{amsrefs}

\usepackage{blkarray}

\numberwithin{equation}{section}
\usetikzlibrary{positioning}
\usetikzlibrary{arrows}

\newtheorem{theorem}{Theorem}[section]
\newtheorem{lemma}[theorem]{Lemma}
\newtheorem{proposition}[theorem]{Proposition}
\newtheorem{corollary}[theorem]{Corollary}

\newtheorem{question}[theorem]{Question}

\theoremstyle{definition}
\newtheorem{definition}[theorem]{Definition} 
\newtheorem{remark}[theorem]{Remark}
\newtheorem{example}[theorem]{Example}
\newtheorem{acknowledgement}{Acknowledgement}

\newcommand{\K}{\mathbb{K}}
\newcommand{\C}{\mathbb{C}}
\newcommand{\N}{\mathbb{N}}

\newcommand{\R}{\mathbb{R}}
\newcommand{\Q}{\mathbb{Q}}
\newcommand{\Pp}{\mathbb{P}}
\newcommand{\m}{\mathfrak{m}}
\newcommand{\supp}{\text{supp}}
\newcommand{\tuple}[1]{\langle #1\rangle}
\newcommand{\st}{\colon}

\newcommand{\p}{\mathfrak{p}}

\DeclareMathOperator{\rk}{\text{rank}~}

\DeclareMathOperator{\HS}{HS}
\DeclareMathOperator{\HF}{HF}
\DeclareMathOperator{\lk}{link}
\DeclareMathOperator{\Soc}{Soc}

\DeclareMathOperator{\Ass}{Ass}

\DeclareMathOperator{\hesd}{hesd}

\newcommand{\qand}{\quad \mbox{and} \quad}

\newcommand{\qfor}{\quad \mbox{for} \quad}

\newcommand{\qforevery}{\quad \mbox{for every} \quad}
\newcommand{\qforall}{\quad \mbox{for all} \quad}

\newcommand{\qif}{\quad \mbox{if} \quad}

\newcommand{\F}{\mathcal{F}}

\newcommand{\e}{\varepsilon}
\newcommand{\ba}{\mathbf{a}}
\newcommand{\bb}{\mathbf{b}}
\newcommand{\be}{\mathbf{e}}

\begin{document}
 
\title{From points to complexes: a concept of unexpectedness for simplicial complexes}

\author[T. Holleben]{Thiago Holleben}
\address[T. Holleben]
{Department of Mathematics \& Statistics,
Dalhousie University,
6297 Castine Way,
PO BOX 15000,
Halifax, NS,
Canada B3H 4R2}
\email{hollebenthiago@dal.ca}

\keywords{Systems of parameters, Lefschetz properties, Stanley-Reisner rings}
\subjclass[2020]{13E10, 13F55, 05E45, 05E40}

 
\begin{abstract}
    In 2018, Cook, Harbourne, Migliore and Nagel introduced the concept of unexpected hypersurfaces, which connects the study of Lefschetz properties of artinian algebras defined by powers of linear forms, to a family of interpolation problems. 
    In this paper, inspired by the theory of unexpected hypersurfaces, we introduce the concept of unexpected systems of parameters for squarefree monomial ideals. Similarly to the setting of points, we show that the existence of an unexpected system of parameters causes a certain algebra to fail the weak Lefschetz property. We then explore combinatorial interpretations of unexpected systems of parameters, and show that this notion is intrinsically related to the theory of balanced complexes. 

    A consequence of our results is that the theory of Rees algebras turns out to be a powerful tool for studying the existence of systems of parameters satisfying special properties. 
\end{abstract}

\maketitle


\section{Introduction}
   
\subsection{Simplicial complexes in commutative algebra}\label{sub:intro1} In 1975, Stanley~\cite{S1975} showed the $h$-polynomial of a simplicial complex $\Delta$ is the numerator of the Hilbert series of an algebra, which is nowadays called the \emph{face ring} or \emph{Stanley-Reisner ring} of $\Delta$. In view of this observation, several tools from commutative algebra became available for the study of inequalities of $f$- and $h$-vectors of simplicial complexes.
 
The earliest application of the idea above can be found already in~\cite{S1975}, where Stanley proves the Upper Bound Conjecture by applying a characterization of Cohen-Macaulay complexes by Reisner~\cite{R1976}. Since then, several results have been shown relying on this idea, including lower bound theorems~\cite{MN2013} and the necessity part of McMullen's $g$-conjecture~\cite{S80,A2018,PP2020,APP2021}.
A property of Cohen-Macaulay rings that plays a key role in these applications is that for an infinite field $\K$ and an $(d+1)$--dimensional Cohen-Macaulay $\K$--algebra $A$, there exists sequences of elements $\theta_1, \dots, \theta_{d+1}$ such that $A/(\theta_1, \dots, \theta_{d+1})$ is a finite dimensional $\K$--vector space. Such sequences of elements are called \emph{systems of parameters} (abbreviated sop).
For the purpose of the applications mentioned, it is often the case that taking a sop given by general linear forms is enough. One of the main advantages of assuming the elements $\theta_1, \dots, \theta_{d+1}$ are general, is that the same procedure for generating the forms works for arbitrary $(d+1)$--dimensional Cohen-Macaulay rings, even though it is not always the same set of forms. 

It is often the case, however, that one wants to prove finer inequalities for special classes of complexes. In order to work in a more specialized setting, a special sop that carries more information on the combinatorics of the object is usually extremely useful. A famous example of such phenomenon is the study of \emph{balanced complexes}.

In 1979, Stanley~\cite{S1979} introduced the class of $d$-dimensional (Cohen-Macaulay) balanced complexes as the class of (Cohen-Macaulay) complexes whose vertices can be colored by $d+1$ colors such that every maximal face of the complex contains exactly one vertex of each color. This class is sometimes called strongly balanced in the literature. The same property that defines balanced complexes naturally defines a set of linear forms that Stanley showed is a sop for the face ring of the underlying complex. Given a balanced complex $\Delta$, this sop is now known as the \emph{colored sop} (of the face ring) of $\Delta$. Over the years, the colored sop (and its variations) have been extensively applied to prove "balanced" versions of known results that hold in a more general setting for complexes that are for example the boundary complex of some simplicial polytope, or homeomorphic to a sphere~\cite{JKM2018,O2024,CJKMN2018}.

Even though special sops are extremely useful in many areas of combinatorics that use tools from commutative algebra, finding specific sops can be a very hard task. This is especially noticeable in the literature, where one easily finds several tools for computing related invariants of monomial ideals (such as their depth), but results around generating regular sequences are extremely rare~\cite{HM2021,FHM2020}.

In this paper we focus on the problem of determining systems of parameters for Cohen-Macaulay or Gorenstein squarefree monomial ideals satisfying special properties. More specifically, for a Cohen-Macaulay (or Gorenstein) simplicial complex $\Delta$ with Stanley-Reisner ideal $I_\Delta \subset R = \K[x_1, \dots, x_n]$, we study the following question:

\begin{question}\label{q:1}
     Given two integers $a,t > 0$ and a homogeneous polynomial $f \in R$, when can we guarantee the existence (or nonexistence) of a system of parameters $\theta_1, \dots, \theta_{d+1}$ of $I_\Delta$ over a field of characteristic zero such that:
     \begin{enumerate}
        \item $x_1^a, \dots, x_n^a, f \in I_\Delta + (\theta_1, \dots, \theta_{d+1})$, and 
        \item $\deg \theta_1 + \dots + \deg \theta_{d+1} + \deg h_\Delta(x) - d - 1 = t$,
     \end{enumerate}
     where $h_\Delta(x)$ is the $h$-polynomial of $\Delta$.
\end{question}

Throughout this paper, we consider~\cref{q:1} where $f = x_1 + \dots + x_n$ and with the following extra condition on the degree $t$ and $t-1$ graded pieces of the ring $A_\Delta(a) = \frac{R}{I_\Delta + (x_1^a, \dots, x_n^a)}$

\begin{enumerate}
    \item[(3)] $\displaystyle \dim_\K A_\Delta(a)_{t-1} \geq \dim_\K A_\Delta(a)_{t}$.
\end{enumerate}

Condition $(3)$, together with our chosen $f$ makes it so general sops will usually not have properties $(1)$ and $(2)$ from~\cref{q:1}, so that we have to look for special sops. Note that conditions $(1)$, $(2)$ and $(3)$ can be interpreted as: $A = \frac{R}{I_\Delta + (\theta_1, \dots, \theta_{d+1})}$ has a high socle degree by $(2)$, and the standard monomials of $A$ do not have high powers by $(1)$ and $(3)$.

Our motivation for looking at~\cref{q:1} with $f = x_1 + \dots + x_n$ comes from an analogy between the recently developed notion of unexpected hypersurfaces~\cite{CHMN2018} and the monomial setting, which we now explain.

\subsection{Points in projective space, Lefschetz properties and unexpectedness}

Given a set of points $X_1, \dots, X_s$ in projective space and integers $\alpha_1, \dots, \alpha_s > 0$, a classical problem in algebraic geometry is to study hypersurfaces of a given degree passing through $X_i$ with multiplicity at least $\alpha_i$. Algebraically, this problem can be seen as studying the graded pieces of the ideal $I = I_{X_1}^{\alpha_1} \cap \dots \cap I_{X_s}^{\alpha_s}$, where $I_{X_i}$ is the defining ideal of the point $X_i$.

One of the many questions that one can ask regarding the interpolation problem above (and its variations) is the dimension of the space of forms of degree $d$ that vanish at points $X_1, \dots, X_s$ and at an extra general point $P$ with multiplicity $m$. In this case, Cook, Harbourne, Migliore and Nagel~\cite{CHMN2018} introduced the notions of "expected" and "actual" dimensions for this space of forms, and introduced the concept of "unexpectedness" exactly when the expected and actual dimension do not match. More concretely, a set of points $X = \{X_1, \dots, X_s\} \subset \Pp^n$ admits an~\emph{unexpected hypersurface} (or \emph{unexpected curve}, if $n = 2$) of degree $d$ with a general point $P$ of multiplicity $m$ if the actual dimension $\dim (I_X \cap I_P^m)_d$ is strictly larger than the expected dimension $\max (0, \dim (I_X)_d - \binom{m + n - 1}{n})$.

Since the definition of unexpected hypersurfaces in 2018, several connections to other areas of mathematics have been found. The connection most relevant to us was in fact already mentioned in~\cite[Theorem 7.5]{CHMN2018}. In 1995 Emsalem and Iarrobino~\cite{EI1995} showed that in order to study the dimension of graded pieces of ideals of the form $I_{X_1}^{\alpha_1} \cap \dots \cap I_{X_s}^{\alpha_s}$, one may use the theory of inverse systems to study graded pieces of an algebra defined by an ideal generated by powers of linear forms that are dual to the points $X_1, \dots, X_s$ instead. The authors in~\cite{CHMN2018} then use Emsalem and Iarrobino's results to show that the existence of an unexpected curve can be understood from the multiplication maps of an artinian algebra, connecting the notion of unexpected hypersurfaces to the theory of Lefschetz properties.

A standard graded artinian algebra $A$ is said to satisfy the \emph{weak Lefschetz property} (abbreviated WLP) if the multiplication maps $\times L: A_i \to A_{i + 1}$ have full rank for every $i \geq 0$, where $L$ is a general linear form. Moreover, $A$ is said to satisfy the \emph{strong Lefschetz property} (abbreviated SLP) if the multiplication maps $\times L^j: A_i \to A_{i + j}$ have full rank for every $i, j \geq 0$. The theory of Lefschetz properties has its roots in the study of cohomology rings of smooth irreducible complex projective varieties, but has deep implications to combinatorics. As an example, if $A$ satisfies the WLP, then the coefficients $h_i$ of its Hilbert series form a unimodal sequence: there exists an integer $k$ such that $h_0 \leq \dots \leq h_k \geq \dots \geq h_d$, where $h_t = 0$ for every $t \geq d$. The theory of Lefschetz properties of Stanley-Reisner rings is in fact the most important tool in the proof of most of the results mentioned in~\cref{sub:intro1}.

In this paper, we replace the linear forms from~\cite[Theorem 7.5]{CHMN2018} by an artinian monomial ideal of the form $J = I_\Delta + (x_1^{a_1}, \dots, x_n^{a_n})$, where $I_\Delta$ is the Stanley-Reisner ideal of a Cohen-Macaulay complex $\Delta$. In~\cref{p:unexpectedfailure} we show that the existence of a system of parameters $\theta_1, \dots, \theta_{d+1}$ satisfying conditions $(1), (2), (3)$ from~\cref{q:1} implies the algebra defined by $J$ fails the WLP due to surjectivity, and failure is caused exactly by the elements of the inverse system of $I_\Delta + (\theta_1, \dots, \theta_{d+1})$. Because of this failure and~\cite[Theorem 7.5]{CHMN2018}, we may view the inverse systems of such sops as analogues of unexpected hypersurfaces. We call such sops \emph{unexpected systems of parameters} of $\Delta$.

Although at first it is not clear which combinatorial properties of $\Delta$ are related to the existence of unexpected sops, one of our goals is to find families of unexpected sops of special families of complexes. We show the following relation between balanced complexes and unexpected linear sops.

\begin{theorem}[\cref{t:coloredunexpected,c:unexpectedcolored}]\label{t:coloredunexpectedintro}
    Let $\Delta$ be a simplicial complex homeomorphic to a $d$-dimensional sphere, where $d > 0$. Then $\Delta$ has an unexpected linear sop if and only if $\Delta$ is balanced. Moreover, in this case the colored sop of $\Delta$ is the unique unexpected linear sop of $\Delta$ up to equality of ideals.
\end{theorem}

In order to prove~\cref{t:coloredunexpectedintro} we first combine a result of Joswig~\cite{J2002} on the facet-ridge graph of a complex homeomorphic to a sphere together with a recent result of Dao and Nair~\cite{DN2024} on the Lefschetz properties of special monomial ideals to compute the inverse system of the colored sop of a balanced sphere.

\begin{lemma}[\cref{l:macaulaydualcolored}]
    Let $\Delta$ be a simplicial complex homeomorphic to a $d$-dimensional sphere, where $d > 0$, and assume $\Delta$ is a balanced complex with colored sop $\theta_1, \dots, \theta_{d+1}$. Then the Macaulay dual generator of $I_\Delta + (\theta_1, \dots, \theta_{d+1})$ is given by 
    $$
        \sum_{F \in B_1} x_F - \sum_{G \in B_2} x_G,
    $$
    where $x_F = \prod_{i \in F} x_i$ and $B_1 \cup B_2$ is a bipartition of the facet-ridge graph of $\Delta$. 
\end{lemma}

Colored sops are one of two known families of systems of parameters for large families of spheres. The other family is often called the universal sop~\cite{HM2021,AR2023} (see also~\cite{S1990,GS1984,DCEP1982}), and comes from elementary symmetric polynomials. The universal sop was shown to be unexpected in~\cite{H2025B} (see~\cref{t:universalunexpected}).

\subsection{Interpolation problems, half-hollow edgewise subdivisions and analytic spread}

Given an integer $m$, one setting where the interpolation problem has been extensively explored is the study of hypersurfaces passing through a set of points $X_1, \dots, X_s$ with multiplicity at least $m$. This setting is one of the main motivations for the study of \emph{symbolic powers} in commutative algebra. In both settings that are relevant to us, when $I = I_{X_1} \cap \dots \cap I_{X_s}$ and when $I$ is a squarefree monomial ideal, the $m$-th symbolic power of $I$ can be computed by 
$$
    I^{(m)} = \bigcap_{\p \in \Ass(I)} \p^m. 
$$ 
In the monomial setting, it is known (see~\cite{MV2012}) that a sufficient condition to guarantee $I^m \neq I^{(m)}$ is to understand the analytic spread $\ell(I)$ of $I$. For similar connections in the setting of points see for example~\cite{HKZ2022}. In~\cite{H2024,H2025} the author used well known interpretations of the analytic spread of an equigenerated monomial ideal as the rank of a matrix to relate the analytic spread of a squarefree monomial ideal (and hence its symbolic powers) to the study of Lefschetz properties. In view of the analogies mentioned previously between interpolation problems of points and Lefschetz properties, the use of analytic spread techniques to understand Lefschetz properties of monomial ideals should be seen as another aspect of the analogy between the monomial setting and the setting of ideals of points.

Our second goal in this paper, is to explore settings where we can guarantee the nonexistence of unexpected systems of parameters. In this direction, we show in~\cref{p:highermatrices} that for a sop satisfying $(1), (2), (3)$ in~\cref{q:1} to exist, a special squarefree monomial ideal must not have maximal analytic spread. In particular, our results show that the theory of Rees algebras can be a powerful tool for finding special regular sequences of Cohen-Macaulay squarefree monomial ideals. 
The notions we use to prove our results are the incidence complexes introduced in~\cite{H2024}, together with special subcomplexes of subdivisions of simplicial complexes. 

Given a simplicial complex $\Delta$, in 2000 Edelsbrunner and Grayson~\cite{EG2000} introduced a complex called the edgewise subdivision of $\Delta$. This construction has since then appeared in many distinct scenarios and is often viewed as the algebraic version of barycentric subdivision, in part due to its connections to Veronese algebras of face rings~\cite{BR2005,KW2012,BW2009}. In this paper, we show that one can understand the rank of multiplication maps of artinian monomial algebras by understanding the analytic spread of a special subcomplex of edgewise subdivisions of incidence complexes. This complex appeared before in the context of hypergraph theory~\cite{MV2015} (in the case of a simplex). We call this construction the half-hollow edgewise subdivision. 

As a corollary, we are able to exploit topological and combinatorial properties of a complex $\Delta$ to guarantee the nonexistence of unexpected sops.

\begin{corollary}[\cref{c:nonexistence}]
    Let $\Delta$ be a $d$-dimensional Cohen-Macaulay collapsible complex with Stanley-Reisner ideal $I_\Delta$. Then for a field $\K$ of characteristic zero and a fixed $a > 0$, $I_\Delta$ does not admit a sop as in~\cref{q:1}, where $t > d(a-1)$.
\end{corollary}

\section{Preliminaries}

\subsection{Simplicial complexes}
A \emph{simplicial complex} on vertex set $V$ is a collection of subsets $\Delta$ of $V$ such that $\tau \subset \sigma \in \Delta$ implies $\tau \in \Delta$. Elements of $\Delta$ are called \emph{faces}, maximal faces are called \emph{facets} and we write $\Delta = \tuple{F_1, \dots, F_s}$ when the facets of $\Delta$ are $F_1, \dots, F_s$. The \emph{dimension} of a face $\sigma \in \Delta$ is $\dim \sigma = |\sigma| - 1$ and the dimension of $\Delta$ is $\dim \Delta = \max(\dim \sigma \st \sigma \in \Delta)$, $1$-dimensional complexes are called \emph{graphs}. We say $\Delta$ is \emph{pure} if every facet of $\Delta$ has the same dimension. Faces of dimension $0$ are called \emph{vertices}, dimension $1$ faces are called \emph{edges} and for a pure $d$-dimensional complex, the $(d-1)$--faces are called \emph{ridges}. The $f$-vector of $\Delta$ is the sequence $f(\Delta) = (f_0(\Delta), \dots, f_d(\Delta))$, where $f_i(\Delta)$ denotes the number of $i$-dimensional faces of $\Delta$. The $h$-vector $(h_0(\Delta), \dots, h_{d+1}(\Delta))$ of a $d$-dimensional complex $\Delta$ is the vector obtained from $f(\Delta)$ by the polynomial relations 
$$ 
    \sum_{i = 0}^{d + 1} f_{i - 1}(\Delta)(x - 1)^{d + 1 - i} = \sum_{i = 0}^{d + 1} h_i(\Delta) x^{d + 1 - i}.
$$
We similarly define the $h$-polynomial $h_\Delta(x)$ of $\Delta$ by taking the $h_i$ to be the coefficient of $x^i$. If it is clear from the context, we will ommit $\Delta$ when mentioning entries of the $f$- and $h$-vectors of $\Delta$. Given a face $\sigma \in \Delta$, the \emph{link} of $\sigma$ in $\Delta$ is the complex $\lk_\Delta(\sigma) = \{\tau \in \Delta \st \tau \cup \sigma \in \Delta \mbox{ and } \tau \cap \sigma = \emptyset\}$. A simplicial complex $\Delta$ is \emph{Cohen-Macaulay} over a field $\K$ if for every $\sigma \in \Delta$ it satisfies 
$$
    \dim \tilde H_i(\lk_\Delta(\sigma); \K) = 0 \qforall i < \dim \lk_\Delta(\sigma),
$$
where $\tilde H_i(-; \K)$ denotes reduced simplicial homology with coefficients in $\K$. Note that a graph is Cohen-Macaulay if and only if it is connected. A simplicial complex $\Delta$ is a \emph{pseudomanifold} if it satisfies the following properties:

\begin{enumerate}
    \item $\Delta$ is pure,
    \item $\Delta$ is \emph{strongly connected}, in other words, for any two facets $F, G$ of $\Delta$ there exists a sequence of facets $F_1, \dots, F_s$ such that $F_1 = F$, $F_s = G$ and $F_{i} \cap F_{i + 1}$ is a ridge of $\Delta$ for every $i$, and 
    \item every ridge of $\Delta$ is contained in at most $2$ facets of $\Delta$.
\end{enumerate}
When $\Delta$ is a pseudomanifold, the simplicial complex $\partial \Delta$ with the facets being the ridges of $\Delta$ contained in only one facet of $\Delta$ is called the \emph{boundary complex} of $\Delta$. If $\partial \Delta = \emptyset$ we say $\Delta$ is a pseudomanifold without boundary, and with boundary otherwise. A $d$-dimensional pseudomanifold without boundary $\Delta = \tuple{F_1, \dots, F_s}$ is \emph{orientable over a field $\K$} if $\tilde H_d(\Delta; \K) = \K$. In this case, an element $\e = \e_1 F_1 + \dots + \e_s F_s \neq 0$ in $\tilde H_d(\Delta; \K)$ is called an \emph{orientation} of $\Delta$. Examples of orientable pseudomanifolds are \emph{simplicial spheres}, that is, simplicial complexes homeomorphic to spheres. 

A simplicial complex $\Delta$ is a \emph{$\K$-homology sphere} if
$$
    \tilde H_i(\lk_\Delta(\sigma); \K) \cong \begin{cases}
        \K \qif i = \dim \lk_\Delta(\sigma) \\
        0 \quad \mbox{otherwise}
    \end{cases} \qforevery \sigma \in \Delta.
$$
Every simplicial sphere is an example of a $\K$-homology sphere. Other examples include for example triangulations of the Poincaré sphere, which satisfies all the homology conditions but has a nontrivial fundamental group and hence can not be homeomorphic to a sphere.

Given a $d$-dimensional pseudomanifold without boundary $\Delta$ on vertex set $V$, the \emph{facet-ridge graph} of $\Delta$ is the graph $G(\Delta)$ with vertex set the facets of $\Delta$, and two facets $F_1, F_2$ are adjacent in $G(\Delta)$ if and only if $F_1 \cap F_2$ is a ridge of $\Delta$. A coloring of $\Delta$ on $n$ colors is a function $\rho: V \to [n]$ such that if $\{u, v\} \in \Delta$, then $\rho(u) \neq \rho(v)$. If the $1$-skeleton of $\Delta$ is $(d+1)$-colorable we say $\Delta$ is a \emph{balanced complex}. 
We will need the following combination of a result of Klee and Novak~\cite{KN2016}, and a result of Joswig~\cite{J2002} (see also~\cite[p. 22]{KNN2016}). We include a proof of this statement for the sake of completeness, and refer the reader to~\cite{KN2016} for the definition of normal pseudomanifolds that will only appear in the proof below.

\begin{theorem}[{\cite[Lemma 6.2 1.]{KN2016},\cite[Proposition 11]{J2002}}]\label{t:bipartitebalanced}
    Let $\Delta$ be a balanced $\K$-homology sphere. Then the facet-ridge graph of $\Delta$ is bipartite. 
    Moreover, if $\Delta$ is a simplicial sphere, then $\Delta$ is balanced if and only if the facet-ridge graph of $\Delta$ is bipartite.
\end{theorem}
 
\begin{proof}
    The first part of the statement follows directly from~\cite[Lemma 6.2.1]{KN2016} since homology spheres are orientable normal pseudomanifolds. 
    
    For the second part of the statement, notice first that if $d = 1$, then both $\Delta$ and $G(\Delta)$ are cycle graphs on the same number of vertices, and so the result holds. Assume now that $d > 1$ and notice that if $\Delta$ is a simplicial $d$-sphere, by definition it is homeomorphic to a $d$-sphere and hence simply connected. Assuming the facet-ridge graph of $\Delta$ is bipartite, we conclude the facet-ridge graph of $\lk_\Delta(\sigma)$ is bipartite for every $\sigma \in \Delta$. The result then follows directly by~\cite[Proposition 11]{J2002}. 
\end{proof}

For a simplicial complex $\Delta = \tuple{F_1, \dots, F_s}$ with vertex set $[n]$, we may associate to it two ideals:

\begin{enumerate}
    \item The \emph{Stanley-Reisner ideal} of $\Delta$, denoted by $I_\Delta = (\prod_{i \in \sigma} x_i \st \sigma \not \in \Delta) \subset \K[x_1, \dots, x_n]$, and
    \item the \emph{facet ideal} of $\Delta$, denoted by $\F(\Delta) = (\prod_{i \in F_1} x_i, \dots, \prod_{i \in F_s} x_i) \subset \K[x_1, \dots, x_n]$. 
\end{enumerate}

The following result of Stanley~\cite{S1996B} will be useful throughout this paper.

\begin{theorem}[{\cite[Theorem II.5.1]{S1996B}}]
    Let $\Delta$ be a simplicial complex on vertex set $V$ such that $\Delta$ is not a cone, $\K$ a field and $I_\Delta \subset R = \K[x_v \st v \in V]$ its Stanley-Reisner ideal. The following are equivalent
    \begin{enumerate}
        \item $\Delta$ is a Cohen-Macaulay orientable pseudomanifold (without boundary) over $\K$,
        \item $\Delta$ is a $\K$-homology sphere, and 
        \item $R/I_\Delta$ is Gorenstein.
    \end{enumerate}
\end{theorem}

\begin{example}\label{ex:complex}
    Let $I = (x_1 x_2, x_3 x_4, x_5 x_6) \subset \K[x_1,\dots, x_6]$. Then $I = I_\Delta$, where $\Delta$ is the boundary of a $2$-dimensional crosspolytope pictured below. The complex $\Delta$ is a $2$-dimensional simplicial sphere (and in particular an orientable pseudomanifold), its facet-ridge graph is the $1$-skeleton of a cube. In particular, the facet-ridge graph of $\Delta$ is bipartite and by~\cref{t:bipartitebalanced} we know $\Delta$ is balanced. A coloring $\chi$ of $\Delta$ is given by 
    $$
        \chi(i) = \begin{cases}
            1 \qif i \in \{1,2\} \\
            2 \qif i \in \{3,4\} \\
            3 \qif i \in \{5,6\}.
        \end{cases}
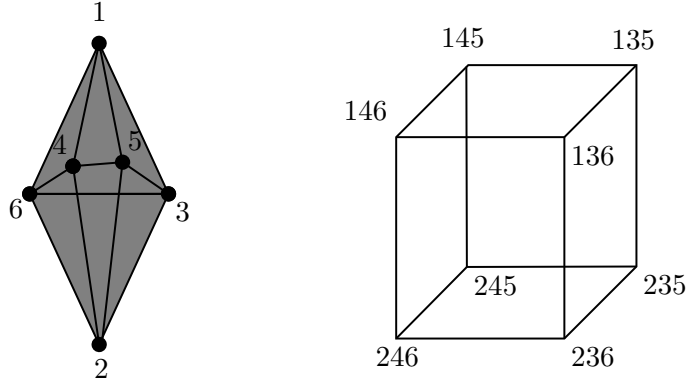
    $$
    \begin{center}

\tikzset{every picture/.style={line width=0.75pt}} 

\begin{tikzpicture}[x=0.75pt,y=0.75pt,yscale=-1,xscale=1]

\draw  [fill={rgb, 255:red, 128; green, 128; blue, 128 }  ,fill opacity=1 ] (136,47) -- (171,123) -- (101,123) -- cycle ;
\draw    (136,47) -- (122.83,109) ;
\draw [color={rgb, 255:red, 0; green, 0; blue, 0 }  ,draw opacity=1 ][fill={rgb, 255:red, 0; green, 0; blue, 0 }  ,fill opacity=1 ]   (136,47) -- (147.83,107) ;
\draw [shift={(136,47)}, rotate = 78.84] [color={rgb, 255:red, 0; green, 0; blue, 0 }  ,draw opacity=1 ][fill={rgb, 255:red, 0; green, 0; blue, 0 }  ,fill opacity=1 ][line width=0.75]      (0, 0) circle [x radius= 3.35, y radius= 3.35]   ;
\draw    (122.83,109) -- (111.6,116.2) -- (101,123.02) ;
\draw [shift={(101,123.02)}, rotate = 147.24] [color={rgb, 255:red, 0; green, 0; blue, 0 }  ][fill={rgb, 255:red, 0; green, 0; blue, 0 }  ][line width=0.75]      (0, 0) circle [x radius= 3.35, y radius= 3.35]   ;
\draw [shift={(122.83,109)}, rotate = 147.33] [color={rgb, 255:red, 0; green, 0; blue, 0 }  ][fill={rgb, 255:red, 0; green, 0; blue, 0 }  ][line width=0.75]      (0, 0) circle [x radius= 3.35, y radius= 3.35]   ;
\draw    (122.83,109) -- (147.83,107) ;
\draw [shift={(147.83,107)}, rotate = 355.43] [color={rgb, 255:red, 0; green, 0; blue, 0 }  ][fill={rgb, 255:red, 0; green, 0; blue, 0 }  ][line width=0.75]      (0, 0) circle [x radius= 3.35, y radius= 3.35]   ;
\draw [shift={(122.83,109)}, rotate = 355.43] [color={rgb, 255:red, 0; green, 0; blue, 0 }  ][fill={rgb, 255:red, 0; green, 0; blue, 0 }  ][line width=0.75]      (0, 0) circle [x radius= 3.35, y radius= 3.35]   ;
\draw    (147.83,107) -- (171,123) ;
\draw [shift={(171,123)}, rotate = 34.63] [color={rgb, 255:red, 0; green, 0; blue, 0 }  ][fill={rgb, 255:red, 0; green, 0; blue, 0 }  ][line width=0.75]      (0, 0) circle [x radius= 3.35, y radius= 3.35]   ;
\draw  [fill={rgb, 255:red, 128; green, 128; blue, 128 }  ,fill opacity=1 ] (136.03,199.01) -- (101,123.02) -- (171,123) -- cycle ;
\draw    (122.83,109) -- (136.03,199.01) ;
\draw [shift={(136.03,199.01)}, rotate = 81.66] [color={rgb, 255:red, 0; green, 0; blue, 0 }  ][fill={rgb, 255:red, 0; green, 0; blue, 0 }  ][line width=0.75]      (0, 0) circle [x radius= 3.35, y radius= 3.35]   ;
\draw [shift={(122.83,109)}, rotate = 81.66] [color={rgb, 255:red, 0; green, 0; blue, 0 }  ][fill={rgb, 255:red, 0; green, 0; blue, 0 }  ][line width=0.75]      (0, 0) circle [x radius= 3.35, y radius= 3.35]   ;
\draw    (147.83,107) -- (137.4,198.34) ;
\draw    (101,123.02) -- (171,123) ;
\draw [shift={(171,123)}, rotate = 359.98] [color={rgb, 255:red, 0; green, 0; blue, 0 }  ][fill={rgb, 255:red, 0; green, 0; blue, 0 }  ][line width=0.75]      (0, 0) circle [x radius= 3.35, y radius= 3.35]   ;
\draw [shift={(101,123.02)}, rotate = 359.98] [color={rgb, 255:red, 0; green, 0; blue, 0 }  ][fill={rgb, 255:red, 0; green, 0; blue, 0 }  ][line width=0.75]      (0, 0) circle [x radius= 3.35, y radius= 3.35]   ;
\draw   (285.83,94.35) -- (322.18,58) -- (407,58) -- (407,159.65) -- (370.65,196) -- (285.83,196) -- cycle ; \draw   (407,58) -- (370.65,94.35) -- (285.83,94.35) ; \draw   (370.65,94.35) -- (370.65,196) ;
\draw   (406.97,159.65) -- (321.48,159.87) -- (321.22,58.77) ;
\draw    (321.48,159.87) -- (285.83,196) ;

\draw (131,25) node [anchor=north west][inner sep=0.75pt]   [align=left] {$\displaystyle 1$};
\draw (132,205) node [anchor=north west][inner sep=0.75pt]   [align=left] {$\displaystyle 2$};
\draw (173,126) node [anchor=north west][inner sep=0.75pt]   [align=left] {$\displaystyle 3$};
\draw (111,92) node [anchor=north west][inner sep=0.75pt]   [align=left] {$\displaystyle 4$};
\draw (149,90) node [anchor=north west][inner sep=0.75pt]   [align=left] {$\displaystyle 5$};
\draw (89,124) node [anchor=north west][inner sep=0.75pt]   [align=left] {$\displaystyle 6$};
\draw (307,38) node [anchor=north west][inner sep=0.75pt]   [align=left] {$\displaystyle 145$};
\draw (393,39) node [anchor=north west][inner sep=0.75pt]   [align=left] {$\displaystyle 135$};
\draw (258,75) node [anchor=north west][inner sep=0.75pt]   [align=left] {$\displaystyle 146$};
\draw (372.65,97.35) node [anchor=north west][inner sep=0.75pt]   [align=left] {$\displaystyle 136$};
\draw (323.48,162.87) node [anchor=north west][inner sep=0.75pt]   [align=left] {$\displaystyle 245$};
\draw (408.97,162.65) node [anchor=north west][inner sep=0.75pt]   [align=left] {$\displaystyle 235$};
\draw (372.65,199) node [anchor=north west][inner sep=0.75pt]   [align=left] {$\displaystyle 236$};
\draw (274,199) node [anchor=north west][inner sep=0.75pt]   [align=left] {$\displaystyle 246$};

\end{tikzpicture}
    \end{center}
\end{example}

\subsection{Artinian algebras and Macaulay duality}

Let $R = \K[x_1, \dots, x_n]$ where $\K$ is a field, $\m = (x_1, \dots, x_n)$ the maximal homogeneous ideal and $I$ an artinian ideal of $R$. The \emph{socle} of the algebra $A = R/I$ is the vector space $\Soc A = \{f \st \m f = 0\}$, and the \emph{socle degree} of $A$ is $\max (d \st A_d \neq 0)$. We say $A$ is \emph{level} if $\Soc A = A_d$, where $d$ is the socle degree of $A$.
Throughout this paper, we denote by $\HF(i, A) = \dim_{\K} A_i$ the \emph{Hilbert function} of an $\K$--algebra $A$, and by $\HS(x, A) = \sum_{i = 0}^\infty x^i \HF(i, A)$ the \emph{Hilbert series} of $A$. It is a well-known fact that for a $d$-dimensional complex $\Delta$, the following equality holds:
$$ 
    \HS(x, R/I_\Delta) = \frac{h_\Delta(x)}{(1 - x)^{d + 1}}.
$$
Let $R/I$ be a $(d+1)$--dimensional algebra. A sequence of homogeneous elements $\theta_1, \dots, \theta_{d+1}$ is called a \emph{system of parameters} (sop for short) of $I$ if $\frac{R}{I + (\theta_1,\dots,\theta_{d+1})}$ is an artinian algebra, that is, a finite dimensional vector space. If moreover $\theta_1, \dots, \theta_{d+1}$ are all linear forms, we say the sequence is a \emph{linear system of parameters} (lsop for short). When the underlying field $\K$ is infinite, one can always guarantee the existence of a lsop of $I$.

Given an integer $a$ we denote by $A_\Delta(a)$ the artinian monomial algebra 
$$
    \frac{\K[x_1, \dots, x_n]}{I_\Delta + (x_1^a, \dots, x_n^a)}.
$$
Several algebraic properties of $A_\Delta(a)$ can be obtained from combinatorial properties of $\Delta$. For example the algebra $A_\Delta(a)$ is level if $\Delta$ is pure, and in this case the socle degree of $A_\Delta(a)$ is given by $(d+1)(a-1)$, where $d = \dim \Delta$~\cite{VTZ2010}.

Let $S = \K[y_1, \dots, y_n]$ and $R = \K[x_1, \dots, x_n]$ be two polynomial rings where $\deg x_i = -\deg y_i = 1$. Let $x^{\ba} = x_1^{a_1}\dots x_n^{a_n} \in R$ and $y^{\bb} = y_1^{b_1} \dots y_n^{b_n} \in S$, the \emph{contraction} of $y^{\bb}$ by $x^{\ba}$ is defined as follows:
$$
    x^{\ba} \circ y^{\bb} = \begin{cases}
    y_1^{b_1 - a_1} \dots y_n^{b_n - a_n} \qif a_i \leq b_i \qforall i \\
    0 \quad \mbox{otherwise}.        
    \end{cases}
$$
By extending the action defined above linearly to arbitrary polynomials in $R$ and $S$ we may view $S$ as an $R$-module. For a homogeneous artinian ideal $I$, the \emph{inverse system} of $I$ is the $R$-module 
$$
    I^{-1} = \{F \st g \circ F = 0 \qforall g \in I\} \subset S.
$$  
When $I^{-1}$ is generated by $F_1, \dots, F_s$, we write $I^{-1} = (F_1, \dots, F_s)$. We note that when the characteristic of the base field is $0$ we may also use differentiation as an action: $x^{\ba} \bullet y^{\bb} = \frac{\partial y^{\bb}}{\partial y^{\ba}}$. In this case we denote by $I^\perp = \{F \st g \bullet F = 0 \qforall g \in I\} \subset S$ the inverse system of $I$ with respect to the action $\bullet$.\looseness-1

The following result due to Macaulay~\cite{macaulaybookinversesystems} is the main result we use throughout the paper. The analogue result using differentiation instead of contraction holds when we replace $I^{-1}$ by $I^\perp$. For a more modern approach to the topic see for example~\cite{lefschetzbook}.

\begin{theorem}[\emph{Macaulay duality}]\label{t:macaulay}
    Let $I$ be an ideal of $R$. Then 
    \begin{enumerate}
        \item $I$ is artinian if and only if $I^{-1}$ is finitely generated,
        \item the algebra $R/I$ is artinian Gorenstein if and only if $I^{-1}$ is a cyclic module, and
        \item there is a bijection (up to multiplication by constants) between artinian Gorenstein algebras $R/I$ of socle degree $d$ and homogeneous polynomials $F$ of degree $d$ given by $I^{-1} = (F)$.
    \end{enumerate}
    Moreover, when $I$ is artinian we have $\dim (R/I)_i = \dim (I^{-1})_{-i}$ for every $i$. 
\end{theorem}

When $R/I$ is an artinian Gorenstein algebra, the polynomial $F \in S$ such that $I^{-1} = (F)$ is called the \emph{Macaulay dual generator} of $R/I$. Even though $F$ is an element of $S$, we may view it as an element in $R$ by simply replacing every $y_i$ by $x_i$. We will often do this throughout the next sections.

\begin{example}
    Let 
    $$
        F = x_1 x_4 x_5 - x_1 x_3 x_5 + x_1 x_3 x_6 - x_1 x_4 x_6 - x_2 x_4 x_5 + x_2 x_4 x_6 - x_2 x_3 x_6 + x_2 x_3 x_5 \in \K[x_1,\dots, x_6].
    $$
    The polynomial $F$ is defined by a bipartition of the facet-ridge graph of the complex $\Delta$ from~\cref{ex:complex} and is the Macaulay dual generator of the algebra 
    $$
        A_F = \frac{R}{I_\Delta + (x_1 + x_2, x_3 + x_4, x_5 + x_6)} \cong \frac{R}{(x_1^2, x_3^2, x_5^2, x_2, x_4, x_6)} \cong \frac{\K[z_1, z_2, z_3]}{(z_1^2, z_2^2, z_3^2)}.
    $$
    In particular, $F$ is the Macaulay dual generator of the complete intersection $I_\Delta + (x_1 + x_2, x_3 + x_4, x_5 + x_6)$. Note that $F$ can be obtained from $x_1 x_2 x_3$ by the linear change of coordinates
    $$
        x_i \mapsto \begin{cases}
            x_{2i} - x_{2i - 1} \qif i \leq 3 \\
            x_i \quad \mbox{otherwise.}
        \end{cases}
    $$
\end{example}

\section{Unexpectedness and Lefschetz properties}

In this section, we define the main concept introduced in this paper: unexpected systems of parameters of squarefree monomial ideals. As is mentioned in the introduction, the main idea behind the definition of these systems of parameters, is to bring recent ideas developed for the study of interpolation problems on sets of points in $\Pp^n$ and unexpected hypersurfaces, to the Stanley-Reisner setting. We will see in later sections that these ideas turn out to give natural combinatorial constructions that have been and still are studied and play key roles in Stanley-Reisner theory.

From the perspective of ideals of points in $\Pp^n$, a result of Iarrobino and Emsalem~\cite{EI1995} uses the theory of inverse systems to show that by applying the usual duality of projective space given by 
$$
    L = a_0 x_0 + \dots + a_n x_n \longleftrightarrow [a_0 : \dots : a_n] \in \Pp^n,
$$
one is able to understand Hilbert functions of ideals defined by powers of linear forms, by studying Hilbert functions of ideals of points with multiplicity. This idea was then exploited in~\cite{CHMN2018} to study when does the multiplication map by (powers of) a general linear form $L$ fails to have full rank in an algebra $A$ defined by powers of linear forms.

Given an artinian algebra $A$ of socle degree $d$, a linear form $L$ is called a \emph{weak Lefschetz element} if the multiplication maps $\times L: A_i \to A_{i + 1}$ have full rank for every $0 \leq i < d$, if $A$ has a Lefschetz element $L$, we say $A$ has (or satisfies) the \emph{weak Lefschetz property} (abbreviated WLP). If the multiplication maps $\times L^j: A_{i} \to A_{i + j}$ have full rank for every $1 \leq j \leq d$ and $0 \leq i \leq d - j$, then we say $L$ is a \emph{strong Lefschetz element} for $A$, and if $A$ has a strong Lefschetz element, we say $A$ has (or satisfies) the \emph{strong Lefschetz property} (abbreviated SLP).
 
Lefschetz properties of artinian algebras have been extensively studied in several different settings. In particular, many techniques have been used to explore the WLP and SLP of artinian algebras of special classes such as the ones defined by monomial ideals, ideals of powers of linear forms, and Gorenstein artinian algebras. In the monomial setting, it is known that a monomial algebra $A$ satisfies the WLP (resp. SLP) if and only if the sum of the variables $L$ is a weak Lefschetz element (resp. strong Lefschetz element)~\cite[Proposition 2.2]{MMN2011}.

In~\cite{H2025B} the author noticed the following connection between Lefschetz properties of artinian monomial algebras and inverse systems.

\begin{theorem}[{\cite[Theorem 5.4]{H2025B}}]\label{t:sopfailure}
    Let $I \subset R = \K[x_1,\dots, x_n]$ be a monomial ideal, $J = I + (x_1^{a_1}, \dots, x_n^{a_n})$, $L = x_1 + \dots + x_n$ and $A = R/J$. Then a multiplication map
    $$
        \times L: A_{i - 1} \to A_i
    $$
    fails to be surjective if and only if there exists a sequence of elements $g_1, \dots, g_k$ such that $I + (g_1, \dots, g_k, L)$ is an artinian ideal, and a polynomial $G \in (I + (L, g_1, \dots, g_k))^{-1}_i$ such that 
    $$
        G = \sum_j c_j x_1^{b_{j_1}}\dots x_n^{b_{j_n}},
    $$
    where $c_j \in \K$ and $b_{j_s} < a_s$ for every $j, s$.
\end{theorem}
 
\cref{t:sopfailure} is the connection between Lefschetz properties and inverse systems that we will constantly use in our setting, instead of the results in~\cite{EI1995}, which are extremely useful in the context of algebras defined by powers of linear forms.

We are now ready to define unexpected systems of parameters.
 
\begin{definition}[\emph{Unexpected systems of parameters}]\label{d:unexpected}
    Let $\Delta$ be a $d$-dimensional Cohen-Macaulay simplicial complex over some field $\K$ on vertex set $[n]$, $f \in R = \K[x_1, \dots, x_n]$ a homogeneous polynomial, $t \geq \deg h_\Delta(x)$ an integer and $\ba = (a_1, \dots, a_n) \in \N^n$. We say a sequence of elements $\theta_1, \dots, \theta_{d+1} \in R$ is an \emph{$(f, \ba)$--unexpected system of parameters of total degree $t$} if:
    \begin{enumerate}[label=({U}{{\arabic*}})]
        \item\label{i:1} $\theta_1, \dots, \theta_{d+1}$ is a homogeneous system of parameters of $I_\Delta$,
        \item\label{i:2} $\deg \theta_1 + \dots + \deg \theta_{d+1} + \deg h_\Delta(x) - d - 1 = t$,
        \item\label{i:3} $f \in I_\Delta + (\theta_1, \dots, \theta_{d+1})$,
        \item\label{i:4} $x_i^{a_i} \in I_\Delta + (\theta_1, \dots, \theta_{d+1})$ for every $i$, and 
        \item\label{i:5} $\HF(t, R/J) \leq \HF(t - \deg f, R/J)$, where $J = I_\Delta + (x_1^{a_1}, \dots, x_n^{a_n})$.
    \end{enumerate}
    When $f$ is the sum of variables of $R$ and $\ba = (a, \dots, a)$, we simply say $\theta_1, \dots, \theta_{d+1}$ is an \emph{$a$-unexpected system of parameters of total degree $t$}.
\end{definition}

The proposition below makes the connection between failure of WLP and existence of unexpected sops concrete in our setting.

\begin{proposition}\label{p:unexpectedfailure}
    Let $\Delta$ be a $d$-dimensional Cohen-Macaulay simplicial complex over a field $\K$ on vertex set $[n]$ and let $L = x_1 + \dots + x_n$. Assume $\theta_1, \dots, \theta_{d+1}$ is an $(L, \ba)$--unexpected sop of $\Delta$ of total degree $t$, where $\ba = (a_1, \dots, a_n)$. Then the monomial algebra 
    $$
        A = \frac{R}{I_\Delta + (x_1^{a_1}, \dots, x_n^{a_n})}
    $$
    fails the WLP due to surjectivity in degree $t-1$.
\end{proposition}

\begin{proof}
    By~\ref{i:2} the socle degree of the algebra $\frac{R}{I_\Delta + (\theta_1, \dots, \theta_{d+1})}$ is $t$. Let $F$ be an element of the inverse system of $I_\Delta + (\theta_1, \dots, \theta_{d+1})$ of degree $t$. By~\ref{i:3} and~\ref{i:4} we know $L \circ F = x_i^{a_i} \circ F = 0$ and so by~\cref{t:sopfailure} we conclude the map 
    $$
        \times L^T: A_{t} \to A_{t - 1}
    $$
    is not injective, where $A = \frac{R}{I_\Delta + (x_1^{a_1},\dots, x_n^{a_n})}$. The result then follows by~\ref{i:5}.
\end{proof}

The example below shows an interesting consequence of the study of unexpected sops: failure of WLP of monomial artinian reductions of monomial complete intersections can be used to find homogeneous polynomials $F$ that are the Macaulay dual generator of complete intersections. 

\begin{example}[\emph{Unexpected sops and Macaulay dual generators of complete intersections}]
    Let $\Delta$ be the boundary of the $4$-dimensional crosspolytope, so that 
    $$
        I_\Delta = (x_1 x_2, x_3 x_4, x_5 x_6, x_7 x_8) \subset \K[x_1,\dots, x_8] = R.
    $$
    Let $L = x_1 + \dots + x_8$ and $K = \ker \times L^T: A_\Delta(3)_6 \to A_\Delta(3)_5$. A Macaulay2~\cite{M2} computation shows $\HF(5, A_\Delta(3)) = 160 \leq 128 = \HF(6, A_\Delta(3))$ and $\dim K = 2$. Moreover, the following two polynomials form a basis of $K$:

    \begin{enumerate}
        \item $F_1 = (x_1 - x_2)(x_3 - x_4)(x_5 - x_6)(x_7 - x_8)(x_1 + x_2 - x_3 - x_4)(x_5 + x_6 - x_7 - x_8)$, and
        \item $F_1 = (x_1 - x_2)(x_3 - x_4)(x_5 - x_6)(x_7 - x_8)(x_1 + x_2 - x_5 - x_6)(x_3 + x_4 - x_7 - x_8)$.
    \end{enumerate}
    
    Again with Macaulay$2$ we see that there are ideals $I_1, I_2$ such that

    \begin{enumerate}
        \item $I_1 = I_\Delta + (L, x_1 + x_2 + x_3 + x_4,x_3^2 + x_4^2, x_7^2 + x_8^2)$, and $I_1^{-1} = (F_1)$
        \item $I_2 = I_\Delta + (L, x_1 + x_2 + x_5 + x_6,x_5^2 + x_6^2, x_7^2 + x_8^2)$, and $I_2^{-1} = (F_2)$,
    \end{enumerate}
    where $F_1$ and $F_2$ are viewed as elements in $R$ instead of $A_\Delta(3)$.

    In particular, one can check that both $L, x_1 + x_2 + x_3 + x_4, x_3^2 + x_4^2, x_7^2 + x_8^2$ and $L, x_1 + x_2 + x_5 + x_6, x_5^2 + x_6^2, x_7^2 + x_8^2$ are $3$-unexpected sops of $\Delta$. Note that since $I_\Delta$ is a monomial complete intersection, both $F_1$ and $F_2$ are the Macaulay dual generators of complete intersections.
\end{example}

We now state the main results of~\cite{H2025B} in the context of unexpected sops. Recall that the $i$-th \emph{elementary symmetric polynomial} $e_i$ of $R = \K[x_1,\dots, x_n]$ is the sum of every squarefree monomail of degree $i$ of $R$.

\begin{theorem}[{\cite[Corollary 4.5 and Theorem 6.6]{H2025B}}]\label{t:universalunexpected}
    Let $\Delta$ be a $d$-dimensional $\K$-homology sphere. Then the elementary symmetric polynomials $e_1, \dots, e_{d+1}$ form an $(d + 2)$--unexpected system of parameters of total degree $\binom{d + 2}{2}$ of $\Delta$.
\end{theorem}

\begin{proof}
    Condition \ref{i:1} is a well-known fact (see for example~\cite{HM2021}), \ref{i:2} and \ref{i:3} follow directly from the definition. Condition \ref{i:4} follows from~\cite[Corollary 4.5]{H2025B} and \ref{i:5} is shown in~\cite[Theorem 6.6]{H2025B}.
\end{proof}

In view of~\cref{p:unexpectedfailure}, the following simple observation from~\cite{H2025B} is the main tool we use for finding unexpected systems of parameters.

\begin{lemma}[{\cite[Lemma 5.3]{H2025B}}]\label{l:contraction}
    Let $f \in R = \K[x_1, \dots, x_n]$, $S = \K[y_1, \dots, y_n]$ and $J \subset R$ a monomial ideal. Then the matrix $M$ that represents the map 
    $$
        \times f^T: R_{i} \to R_{i - \deg f}
    $$
    is the same matrix that represents the map 
    $$
        f \circ: S_{-i} \to S_{\deg f - i}.
    $$ 
    In particular, the matrix that represents the map 
    $$
        \times f^T: (R/J)_{i} \to (R/J)_{i - \deg f}
    $$
    is obtained from $M$ by deleting rows and columns corresponding to monomials in $J$.
\end{lemma}

Throughout the next sections, we will apply the results above to explore the existence (and nonexistence) of unexpected systems of parameters.

\section{$2$-Unexpected systems of parameters and balanced spheres}

We begin by showing that unexpected systems of parameters have appeared before under different contexts in the literature, and some of the questions arising from them (see~\cref{s:questions}) have even been studied before in special cases.

In 1979, Stanley~\cite{S1979} began the study of (Cohen-Macaulay) balanced complexes. One important theme in the study of these complexes is that a coloring $\rho: V \to [d + 1]$ of a balanced $d$-dimensional complex $\Delta$ on vertex set $V$ can be used to study special invariants of $\Delta$, which often correspond to invariants of a finer grading of the ring $R/I_\Delta$, where $\deg x_v = \be_{\rho(v)}$, the $\rho(v)$-canonical basis element of $\R^{d+1}$. The first example of this strategy of using a coloring of $\Delta$ to induce a finer grading on $R/I_\Delta$ can be found in~\cite{S1979}, where Stanley showed the following.

\begin{proposition}[{\cite[Corollary 4.2]{S1979}}]\label{p:coloredsop}
    Let $\Delta$ be a $d$-dimensional Cohen-Macaulay balanced complex on vertex set $V$ and let $\rho: V \to [d+1]$ be a coloring of $\Delta$. Then the set of linear forms 
    $$
        \theta_{i} = \sum_{v \in \rho^{-1}(i)} x_v \qforevery i \in [d+1]
    $$
    is a system of parameters of $I_\Delta$.
\end{proposition}

The system of parameters in~\cref{p:coloredsop} is called the \emph{colored sop} of $\Delta$, with respect to the coloring $\rho$. For a balanced $d$-sphere $\Delta$, the Lefschetz properties of the ring $A = \frac{R}{I_\Delta + (\theta_1, \dots, \theta_{d+1})}$ were studied in~\cite{JKM2018,CJKMN2018}. In particular, in~\cite{CJKMN2018} the authors conjecture that the ring $A$ satisfies the SLP, which they call the \emph{colored SLP} of $\Delta$.

More recently in a different context, in 2024 Dao and Nair~\cite{DN2024} studied of the WLP of algebras of the form $A_\Delta(2) = \frac{R}{I_\Delta + (x_1^2, \dots, x_n^2)}$ from the perspective of Stanley-Reisner theory. They showed the following.

\begin{theorem}[{\cite[Section 4]{DN2024}}]\label{t:daonair}
    Let $\Delta$ be a $d$-dimensional pseudomanifold without boundary, $A = A_\Delta(2)$ and $L = x_1 + \dots + x_n$. Then the multiplication map 
    $$
        \times L: A_{d} \to A_{d + 1}
    $$
    is not surjective if and only if $G(\Delta)$ is bipartite. In particular, the map fails to have full rank if and only if $G(\Delta)$ is bipartite.
\end{theorem}

In view of~\cref{t:bipartitebalanced}, a weaker version of~\cref{t:daonair} could be stated in terms of the balanced property of $\Delta$. We now show that this connection is in fact enough to characterize which simplicial spheres have $2$-unexpected systems of parameters.

\begin{lemma}\label{l:macaulaydualcolored}
    Let $\Delta = \tuple{\sigma_1, \dots, \sigma_s}$ be a balanced $d$-dimensional $\K$-homology sphere on vertex set $[n]$, $\rho$ a coloring of $\Delta$ and $\theta_1, \dots, \theta_{d+1}$ the corresponding colored sop. Then the Macaulay dual generator of $I_\Delta + (\theta_1, \dots, \theta_{d+1})$ is the polynomial 
    $$
        F_{\Delta, \rho} = \sum_{\sigma \in B_1} x_\sigma - \sum_{\sigma \in B_2} x_\sigma, 
    $$
    where $x_\sigma = \prod_{i \in \sigma} x_i$ and $B_1 \cup B_2$ is a bipartition of $G(\Delta)$.
\end{lemma}

\begin{proof}
    First note that since $\Delta$ is a balanced homology sphere,~\cref{t:bipartitebalanced} implies $G(\Delta)$ is bipartite and hence it is posible to partition $V(G(\Delta))$ into two disjoint independent sets $B_1, B_2$. Since we know $I_\Delta + (\theta_1, \dots, \theta_{d+1})$ is an artinian Gorenstein ideal of socle degree $d+1$, we only need to find a homogeneous polynomial $F$ of degree $d+1$ such that $m \circ F = \theta_i \circ F = 0$ for every $i$ and every generator $m$ of $I_\Delta$. The result would then follow since Macaulay duality gives a bijection between artinian Gorenstein algebras of the form $R/I$ and homogeneous polynomials in $R = \K[x_1,\dots, x_n]$.

    By the definition of $F_{\Delta, \rho}$ we have $m \circ F_{\Delta, \rho} = 0$ for every generator $m$ of $I_\Delta$. Hence we only need to show that $\theta_i \circ F_{\Delta, \rho} = 0$ for every $i$. Since the support of every monomial in $F_{\Delta, \rho}$ is a face of $\Delta$, we may view $F_{\Delta, \rho}$ as an element of $A = \frac{R}{I_\Delta + (x_1^2, \dots, x_n^2)}$. Finally, by~\cref{l:contraction} we only need to show $F_{\Delta, \rho}$ is in the kernel of the map $\times \theta_i^T: A_{d + 1} \to A_{d}$ for every $i$.
    
    By the definition of the colored sop of $\Delta$ we have $L = x_1 + \dots + x_n = \theta_1 + \dots + \theta_{d+1}$.
    It is known (see for example~\cite[Lemma 23]{H2024}) that the matrix $M$ representing the map $\times L^T: A_{d+1} \to A_d$ is equal to the signless incidence matrix of $G(\Delta)$. As a consequence, since $G(\Delta)$ is a connected bipartite graph we conclude the kernel of $M$ is spanned exactly by $F_{\Delta, \rho}$ (see for example~\cite[Theorem 8.2.1]{G2001}). Since every facet of $\Delta$ contains exactly one vertex of each color, we conclude for every ridge $\tau$ there exists a unique color $c_\tau$ such that $\tau \cap \rho^{-1}(c_\tau) = \emptyset$. 
    
    Next, if $\times \theta_i^T(F_{\Delta, \rho}) = 0$ for all $i$ we are done. Assume without loss of generality that $\times \theta_1^T(F_{\Delta, \rho}) \neq 0$, and let $\tau$ be a ridge of $\Delta$ such that $x_\tau = \prod_{j \in \tau} x_j$ is a monomial of $\times \theta_1^T(F_{\Delta, \rho})$ with a nonzero coefficient. Since $\tau$ contains vertices of every color other than $1$, we conclude the coefficient of $x_\tau$ in $\times \theta_i^T(F_{\Delta, \rho})$ is zero for every $i \neq 1$. This is a contradiction since
    $$
        \times \theta_1^T(F_{\Delta, \rho}) + \dots + \times \theta_{d+1}^T(F_{\Delta, \rho}) = \times L^T(F_{\Delta, \rho}) = 0
    $$
    and $\times \theta_1^{T}(F_{\Delta, \rho})$ cannot cancel with any other term in the sum above.
\end{proof}

\begin{remark}\label{r:boundary}
    Although we prove~\cref{l:macaulaydualcolored} without mentioning the last boundary map of $\Delta$, similarly to the proof of~\cite[Theorem 4.2]{H2025B} (see also~\cite[Remark 4.3]{H2025B}), the key point of~\cref{l:macaulaydualcolored} is the fact that $\times L^T: A_{d+1} \to A_d$ is essentially the last boundary map of $\Delta$. In other words, when $\Delta$ is a balanced $\K$-homology sphere, the map $L \circ$ mimics the boundary map of $\Delta$ when it is applied to monomials of the form $x_{\sigma_i} = \prod_{j \in \sigma_i} x_j$.
\end{remark}

We are now ready to show the main result of this section.

\begin{theorem}[\emph{Unexpected systems of parameters and balanced spheres}]\label{t:coloredunexpected}
    Let $\Delta$ be a $d$-dimensional simplicial sphere with $d > 0$. Then $\Delta$ has an $2$-unexpected system of parameters if and only if $\Delta$ is balanced. In this case, the colored sop is an example of an $2$-unexpected sop of $\Delta$.
\end{theorem}

\begin{proof}
    Let $\theta_1, \dots, \theta_{d+1}$ be the colored sop of $\Delta$ given by a coloring $\rho$ of $\Delta$ and let $L$ be the linear form given by the sum of variables of $R = \K[x_1, \dots, x_n] \supset I_\Delta$.
    Assume first that $\Delta$ is balanced. Note that $L = \theta_1 + \dots + \theta_{d+1}$, hence $L \in I_\Delta + (\theta_1, \dots \theta_{d+1})$. \cref{l:macaulaydualcolored} directly shows $x_i^2 \in I_\Delta + (\theta_1, \dots, \theta_{d+1})$ for every $i$ since $x_i^2 \circ F_{\Delta, \rho} = 0$. The only missing property to conclude $\theta_1, \dots, \theta_{d+1}$ is an $2$-unexpected sop is that $\HF(d + 1, R/J) \leq \HF(d, R/J)$, where $J = I_\Delta + (x_1^2, \dots, x_n^2)$. This follows since $\HF(i, R/J) = f_{i-1}(\Delta)$ and since $\Delta$ is a $d$-dimensional orientable pseudomanifold without boundary and $d > 0$, a double counting argument gives $f_{d+1} = \frac{2}{d+1}f_d \leq f_{d}$. 

    Assume now that $\Delta$ has an $2$-unexpected system of parameters $\theta_1, \dots, \theta_{d+1}$ of total degree $t \geq d + 1$. Then by definition we know:
    \begin{enumerate}
        \item[\ref{i:1}] $I_\Delta + (\theta_1, \dots, \theta_{d+1})$ is an artinian Gorenstein ideal and hence has a Macaulay dual generator $F$, 
        \item[\ref{i:3}] $L \in I_\Delta + (\theta_1, \dots, \theta_{d+1})$, 
        \item[\ref{i:4}] $x_i^2 \in I_\Delta + (\theta_1, \dots, \theta_{d+1})$, and
        \item[\ref{i:5}] $\HF(t, R/J) \leq \HF(t - 1, R/J)$, where $J = I_\Delta + (x_1^2, \dots, x_n^2)$.
    \end{enumerate}
    By \ref{i:5} above and since the socle degree of $R/J$ is $d+1$ we conclude $t \leq d + 1$, but since by definition we have $t \geq d + 1$, we must have $t = d + 1$.
    Since $\theta_1, \dots, \theta_{d+1}$ is a sop of $I_\Delta$, it is a well-known fact (see for example~\cite[Lemma 3.5]{H2025B}) the Hilbert series of $\frac{R}{I_\Delta + (\theta_1, \dots, \theta_{d+1})}$ is equal to 
    $$
        h_\Delta(x) \prod_{i = 1}^{d+1}(1 + x + \dots + x^{\deg \theta_i - 1}).
    $$ 
    In particular, $\deg h_\Delta(x) + \deg \theta_1 + \dots + \deg \theta_{d+1} - d - 1 = d + 1$, and since $\deg h_\Delta(x) = d + 1$ we conclude $\theta_1, \dots, \theta_{d+1}$ is a linear system of parameters of $I_\Delta$. By~\cref{t:sopfailure} we know the existence of the polynomial $F$ says the algebra $R/J$ fails the WLP in degree $\deg F = d + 1$ due to surjectivity. By~\cref{t:daonair} we conclude the facet-ridge graph of $\Delta$ is bipartite and finally by~\cref{t:bipartitebalanced} we conclude $\Delta$ is balanced.
\end{proof}

Note that a consequence of~\cref{t:coloredunexpected} is that an $2$-unexpected sop must be linear. We can prove the following stronger statement.

\begin{corollary}[\emph{A characterization of unexpected linear systems of parameters}]\label{c:unexpectedcolored}
    Let $\Delta$ be a $d$-dimensional balanced simplicial sphere and $\theta_1', \dots, \theta_{d+1}'$, $\theta_1'', \dots, \theta_{d+1}''$ be two $2$-unexpected sops of $\Delta$. Then 
    $$
        I_\Delta + (\theta_1', \dots, \theta_{d+1}') = I_\Delta + (\theta_1'', \dots, \theta_{d+1}'').
    $$
    In particular, we can take $\theta_1', \dots, \theta_{d+1}'$ to be the colored sop of $\Delta$.
\end{corollary}

\begin{proof}
    Let $A = A_\Delta(2)$ and $L = x_1 + \dots + x_n$. By~\cref{t:daonair} and since the multiplication map $\times L: A_{d} \to A_{d+1}$ is represented by the transpose of the (signless) incidence matrix of the connected graph $G(\Delta)$ (see for example~\cite[Lemma 23]{H2024}), we know the rank of this map is exactly one less than the maximal possible rank~\cite[Theorem 8.2.1]{G2001}. In other words, we have $\dim \ker \times L^T: A_{d+1} \to A_d = 1$. Let $\rho$ be a coloring of the $1$-skeleton of $\Delta$, $\theta_1, \dots, \theta_{d+1}$ the corresponding colored sop and $F_{\Delta, \rho}$ the Macaulay dual generator of $I_\Delta + (\theta_1, \dots, \theta_{d+1})$. By~\cref{t:coloredunexpected} we know this colored sop is an $2$-unexpected sop of $\Delta$, and by~\cref{t:sopfailure} we know $F_\rho$ generates the kernel of $\times L^T: A_{d+1} \to A_d$. Moreover, for every other $2$-unexpected sop of $\Delta$, the Macaulay dual generator $G$ of the corresponding ideal $I$ is also in the same kernel, and as the dimension of the kernel is $1$ we conclude $ G = c F_\rho$ for some constant $c$. Macaulay duality then implies $I_\Delta + (\theta_1, \dots, \theta_{d+1}) = I$.
\end{proof}

\section{Analytic spread, collapsibility and the nonexistence of unexpected sops}\label{s:collapsible}

Our next goal is to develop techniques that will allow us to prove the nonexistence of unexpected sops for special classes of complexes. Some of our results have consequences to other areas of commutative algebra, namely to the theory of Rees algebras of monomial ideals. This connection is in fact one more aspect of the analogies between the geometric study of unexpectedness, and the notion of unexpected systems of parameters for simplicial complexes introduced earlier. 

Given a finite set of points $X$ in the projective space $\Pp_\K^n$, the interpolation problem consists of studying the vector space dimension of polynomials passing through the points in $X$ with prescribed orders. As is mentioned in~\cite{HMN2024}, the concept of unexpectedness originated from studying the interpolation problem and its variations.

In view of the well known Zariski-Nagata theorem, symbolic powers play a key role in understanding the interpolation problem. In a different context, for a monomial ideal $I$, it is known that the Rees algebra of $I$ is an extremely useful tool for understanding the symbolic powers of $I$. In particular, a sufficient condition that guarantees the symbolic powers of $I$ are not equal to the ordinary powers of $I$ involves the \emph{analytic spread} of $I$ (see~\cite[Theorem 4.34, Proposition 4.39]{MV2012}). For the corresponding theory for ideals of points relating symbolic powers and analytic spread, see for example~\cite{HKZ2022}.

In~\cite{H2024, H2025} it is shown that the analytic spread of special squarefree monomial ideals can be used to detect whether an algebra of the form $A_\Delta(2) = \frac{\K[x_1,\dots, x_n]}{I_\Delta + (x_1^2, \dots, x_n^2)}$ either has or fails the WLP. In this section we use combinatorial constructions arising from the edgewise subdivision of a complex, in order to generalize some of the results from~\cite{H2024, H2025}. Along the way, we also show connections between the topological notion of collapsibility, and the theory of Rees algebras of facet ideals.

\subsection{Analytic spread}
We begin by recalling the necessary notions. Let $I$ be an ideal of $S = \K[x_1, \dots, x_n]$. The \emph{Rees ring} of $I$ is the ring $S[It] = \oplus_{n \in \N} I^n t^n$. Let $\m = (x_1, \dots, x_n)$ be the maximal homogeneous ideal of $S$. The \emph{analytic spread} of $I$, denoted by $\ell(I)$ is the Krull dimension of $S[It]/\m S[It]$. The analytic spread of an ideal $I$ determines whether several important invariants of $I$, such as various notions of multiplicities (see for example~\cite{T2001}) are positive. When $I$ is a squarefree monomial ideal, further connections to symbolic powers are known~\cite{MV2012}. In the monomial setting, there is a concrete method for computing the analytic spread of $I$ using the following matrix.
 
\begin{definition}[\emph{Log matrices}]
    Let $I = (m_1, \dots, m_s)$ be a monomial ideal where $m_i = x_1^{a_{i1}} \dots x_n^{a_{in}}$, with $a_{ij} \geq 0$. The \emph{log matrix} of $I$ is denoted by $\log(I)$ and its entries are given by $\log(I)_{ij} = a_{ij}$. Note that the sum of the $i$-th row is equal to $\deg m_i$.
\end{definition}

\begin{lemma}[{\cite[Lemma 3.2]{MMV2012}}]\label{l:analyticspreadrank}
    Let $I$ be a monomial ideal generated by elements of degree $d$. Then $\ell(I) = \rk \log(I)$.
\end{lemma}

In the setting of~\cref{l:analyticspreadrank}, we say $I$ has \emph{maximal analytic spread} if $\log(I)$ has full rank. As we will see, these log matrices will help us determine whether monomial algebras satisfy the WLP, and in view of~\cref{p:unexpectedfailure}, will be used as a tool for proving the nonexistence of unexpected sops of given total degree. 

In~\cite{H2024,H2025} the author introduced the following class of simplicial complexes that will play a key role in this and later sections.

\begin{definition}[\emph{Incidence complexes~\cite{H2024,H2025}}]
    Let $\Delta$ be a simplicial complex of dimension $d$ and $1 \leq i < d$. The $i$-th \emph{incidence complex} of $\Delta$, denoted by $\Delta(i)$ has as its vertex set the set of $(i-1)$--faces of $\Delta$, and a subset $\sigma$ of vertices is a facet of $\Delta(i)$ if and only if the vertices of $\sigma$ are the $(i-1)$--faces of an $i$-face of $\Delta$.
\end{definition}
 
We note that if $F_i, F_j$ are two facets of an incidence complex, then $|F_i \cap F_j| \leq 1$ (see~\cite[Proposition 20.2]{H2024}).
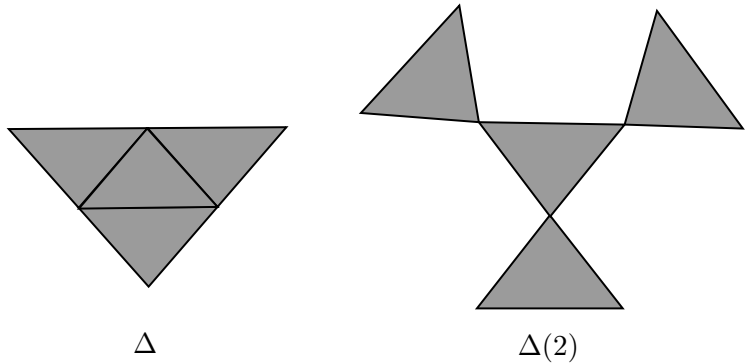
\begin{figure}[h!]

        \tikzset{every picture/.style={line width=0.75pt}} 

        \begin{tikzpicture}[x=0.75pt,y=0.75pt,yscale=-1,xscale=1]
        
        \draw  [fill={rgb, 255:red, 155; green, 155; blue, 155 }  ,fill opacity=1 ] (175.01,131.77) -- (139.74,92.01) -- (209.74,91.54) -- cycle ;
        \draw  [fill={rgb, 255:red, 155; green, 155; blue, 155 }  ,fill opacity=1 ] (105.01,132.24) -- (69.74,92.47) -- (139.74,92.01) -- cycle ;
        \draw  [fill={rgb, 255:red, 155; green, 155; blue, 155 }  ,fill opacity=1 ] (140.27,172.01) -- (105.01,132.24) -- (175.01,131.77) -- cycle ;
        \draw  [fill={rgb, 255:red, 155; green, 155; blue, 155 }  ,fill opacity=1 ] (139.57,92.2) -- (175.01,131.77) -- (105.12,132.63) -- cycle ;
        \draw  [fill={rgb, 255:red, 155; green, 155; blue, 155 }  ,fill opacity=1 ] (342.63,136) -- (379.39,183) -- (305.86,183) -- cycle ;
        \draw  [fill={rgb, 255:red, 155; green, 155; blue, 155 }  ,fill opacity=1 ] (342.81,136.6) -- (306.83,89) -- (380.34,90.22) -- cycle ;
        \draw  [fill={rgb, 255:red, 155; green, 155; blue, 155 }  ,fill opacity=1 ] (380.34,90.22) -- (396.66,32.82) -- (439.98,92.23) -- cycle ;
        \draw  [fill={rgb, 255:red, 155; green, 155; blue, 155 }  ,fill opacity=1 ] (306.83,89) -- (247.33,84.36) -- (296.98,30.14) -- cycle ;
        
        \draw (131,194) node [anchor=north west][inner sep=0.75pt]   [align=left] {$\displaystyle \Delta $};
        \draw (325,194) node [anchor=north west][inner sep=0.75pt]   [align=left] {$\displaystyle \Delta ( 2)$};
        \end{tikzpicture}
        \caption{A complex $\Delta$ on the left, and its incidence complex $\Delta(2)$ on the right.}
\end{figure}

\subsection{Collapsibility} A face $\sigma$ of a simplicial complex $\Delta$ is said to be \emph{free} if it is only contained in one other face $\tau$ of $\Delta$. An \emph{elementary collapse} of $\Delta$ by a free face $\sigma$, is the process of deleting the faces $\sigma, \tau$ from $\Delta$, resulting in a new complex $\Delta'$. In this case we say $\Delta'$ is obtained from $\Delta$ via an elementary collapse. A complex $\Delta$ is said to be \emph{collapsible} if the complex with a single vertex can be obtained from $\Delta$ via a sequence of elementary collapses. 

It is clear that if $\Delta'$ is obtained from $\Delta$ via an elementary collapse, then $\Delta$ is homotopy equivalent to $\Delta'$. This connection between the combinatorial notion of an elementary collapse, and the topological notion of homotopy equivalence is one of the many reasons collapsilibity has attracted interest from researchers in many different aspects of combinatorial topology. It follows for example that a collapsible simplicial complex is contractible. The converse however is far from true.

\begin{example}[Collapsible vs contractible complexes~\cite{Z1964}]
    The complex $\Delta$ below is a triangulation of the \emph{Dunce hat}, one of the earliest examples of a contractible but non collapsible complex.

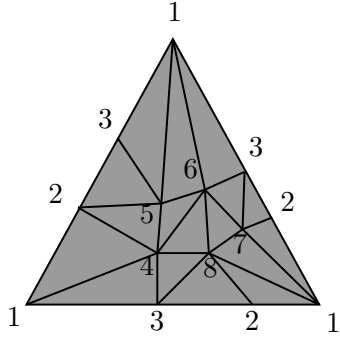
\begin{center}
\tikzset{every picture/.style={line width=0.75pt}} 

\begin{tikzpicture}[x=0.75pt,y=0.75pt,yscale=-1,xscale=1]

\draw  [fill={rgb, 255:red, 155; green, 155; blue, 155 }  ,fill opacity=1 ] (115.17,54) -- (189,188) -- (41.33,188) -- cycle ;
\draw    (41.33,188) -- (107.33,162) ;
\draw    (67.33,139) -- (107.33,162) ;
\draw    (67.33,139) -- (109.33,137) ;
\draw    (87.33,104) -- (109.33,137) ;
\draw    (109.33,137) -- (107.33,162) ;
\draw    (107.33,162) -- (107.33,188) ;
\draw    (107.33,188) -- (133.33,162) ;
\draw    (107.33,162) -- (133.33,162) ;
\draw    (109.33,137) -- (131.33,130) ;
\draw    (131.33,130) -- (150.33,150) ;
\draw    (133.33,162) -- (150.33,150) ;
\draw    (131.33,130) -- (107.33,162) ;
\draw    (131.33,130) -- (133.33,162) ;
\draw    (131.33,130) -- (151.33,121) ;
\draw    (150.33,150) -- (151.33,121) ;
\draw    (131.33,130) -- (115.17,54) ;
\draw    (109.33,137) -- (115.17,54) ;
\draw    (133.33,162) -- (155.33,188) ;
\draw    (133.33,162) -- (189,188) ;
\draw    (150.33,150) -- (189,188) ;
\draw    (150.33,150) -- (165.33,144) ;

\draw (111,34) node [anchor=north west][inner sep=0.75pt]   [align=left] {$\displaystyle 1$};
\draw (30,188) node [anchor=north west][inner sep=0.75pt]   [align=left] {$\displaystyle 1$};
\draw (191,191) node [anchor=north west][inner sep=0.75pt]   [align=left] {$\displaystyle 1$};
\draw (51,126) node [anchor=north west][inner sep=0.75pt]   [align=left] {$\displaystyle 2$};
\draw (150,190) node [anchor=north west][inner sep=0.75pt]   [align=left] {$\displaystyle 2$};
\draw (168,130) node [anchor=north west][inner sep=0.75pt]   [align=left] {$\displaystyle 2$};
\draw (76,88) node [anchor=north west][inner sep=0.75pt]   [align=left] {$\displaystyle 3$};
\draw (152,102) node [anchor=north west][inner sep=0.75pt]   [align=left] {$\displaystyle 3$};
\draw (102,190) node [anchor=north west][inner sep=0.75pt]   [align=left] {$\displaystyle 3$};
\draw (97,162) node [anchor=north west][inner sep=0.75pt]   [align=left] {$\displaystyle 4$};
\draw (97,136) node [anchor=north west][inner sep=0.75pt]   [align=left] {$\displaystyle 5$};
\draw (119,113) node [anchor=north west][inner sep=0.75pt]   [align=left] {$\displaystyle 6$};
\draw (144,151) node [anchor=north west][inner sep=0.75pt]   [align=left] {$\displaystyle 7$};
\draw (129,163) node [anchor=north west][inner sep=0.75pt]   [align=left] {$\displaystyle 8$};

\end{tikzpicture}
\end{center}
\end{example}

\subsection{Analytic spread vs collapsibility} From topological point of view, the key property of collapsible complexes that is useful for us is the fact that the last boundary map of a collapsible complex is upper triangular (after reordering rows and columns). In a more algebraic language, we have the following lemma that allows us to use collapsibility as a tool to compute the analytic spread of incidence complexes.

\begin{lemma} 
    Let $\Delta$ be a pure collapsible $d$-dimensional simplicial complex with $f_d$ facets. Then $\F(\Delta(d))$ has maximal analytic spread, and in particular
    $$
        \ell(\F(\Delta(d))) = f_d.
    $$
\end{lemma}
 
\begin{proof}
    By~\cref{l:analyticspreadrank} we only need to compute the rank of the log matrix $M$ of $\F(\Delta(d))$. Let $f_{d-1}$ be the number of ridges of $\Delta$. Let $\sigma_1, \dots, \sigma_{f_d}$ be a sequence of ridges of $\Delta$ that can be taken as the start of a sequence of elementary collapses starting from $\Delta$. Let $\tau_1$ be the unique face of $\Delta$ that strictly contains $\sigma_1$, and recursively define $\tau_i$ as the unique face of $\Delta \setminus \{\sigma_1, \dots, \sigma_{i-1}, \tau_1, \dots, \tau_{i-1}\}$ that strictly contains $\sigma_i$. Note that the map $\tau_i \mapsto \sigma_i$ is an injection between the set of $d$-faces of $\Delta$ into the set of $(d-1)$--faces of $\Delta$, hence $f_d \leq f_{d-1}$.
     
    In terms of $\Delta(d)$, $\sigma_1$ is a free vertex of $\Delta(d)$, and $\sigma_i$ is a free vertex of $\Delta(d) \setminus \{\sigma_1, \dots, \sigma_{i-1}, \tau_1', \dots, \tau_{i-1}'\}$, where $\tau_j'$ is the facet of $\Delta(d)$ that corresponds to the facet $\tau_j$ of $\Delta$. In particular, consider the minor of $\log(\F(\Delta(d)))$ given by the columns corresponding to $\{\sigma_1, \dots, \sigma_{f_d}\}$, and all the rows. By construction, after permuting rows and columns we obtain an upper triangular matrix with ones in the diagonal, hence this minor guarantees that $\rk \log(\F(\Delta(d))) \geq f_d$. Since $\log(\F(\Delta(d)))$ is a $f_{d}$ by $f_{d-1}$ matrix, we conclude it has full rank. The result then follows by~\cref{l:analyticspreadrank} since $\F(\Delta(d))$ is a monomial ideal generated by elements of degree $d + 1$.
\end{proof}

Our next goal is to show that given a collapsible complex $\Delta$ of dimension $d$, $J = I_\Delta + (x_1^{a}, \dots, x_n^{a})$ and $L = x_1 + \dots + x_n \in R = \K[x_1,\dots, x_n]$, the map $\times L: (R/J)_{d(a-1)} \to (R/J)_{d(a-1) + 1}$ is surjective. The value $(d+1)(a-1)$ is chosen so that we can guarantee that every row of the corresponding matrix adds up to the same number. This property allows us to interpret the matrix as the log-matrix of an equigenerated facet ideal, and allows us to state our results in the context of the theory of Rees algebras. Before proving the statement we need to recall a construction from hypergraph theory and combinatorial topology.

In~\cite{MV2015}, Mirzakhani and Vondr\'ak introduced a family of hypergraphs based on the edgewise subdivision of a simplex~\cite{EG2000}, in order to study problems related to hypergraph colorings. Here we define a slightly more general version of their construction, which we call the \emph{$r$-fold half-hollow edgewise subdivision} of a simplicial complex $\Delta$, denoted by $\hesd(\Delta, r)$. 

Let $r > 0$ be an integer, and denote by 
$$
    \Omega_n^r = \{(a_1, \dots, a_n) \st a_1 + \dots + a_n = r \qand a_i \geq 0\} \in \R^n
$$ 
the set of lattice points in the $r$-th dilation of the standard simplex. Moreover, for a vector $\ba = (a_1, \dots, a_n) \in \R^n$, we denote by $\supp(\ba) = \{i \st a_i \neq 0\}$ the support of $\ba$. Note that for a simplicial complex $\Delta$ on vertex set $[n] = \{1, \dots, n\}$, we may view the vertices of $\Delta$ as the canonical basis $\{\be_1, \dots, \be_n\}$ of $\R^n$. 

Let $\Delta = \tuple{F_1, \dots, F_s}$ be a simplicial complex on vertex set $[n]$ such that $|F_i \cap F_j| \leq 1$. The complex $\hesd(\Delta, r)$ has vertex set $\Omega_n^r$. A collection of vectors $A \subset \Omega_n^r$ is a face of $\hesd(\Delta, r)$ if and only if the following conditions are satisfied:
\begin{enumerate}[label=({HESD}{{\arabic*}})]
    \item\label{hesd1} $\bigcup_{\ba \in A} \supp(\ba) \in \Delta$
    \item\label{hesd2} $A = \{\ba + \be_{i} \st i \in \supp(\ba)\}$ for some fixed vector $\ba = (a_1, \dots, a_n)$, such that $a_1 + \dots + a_n = r - 1$ and $a_i \geq 0$.
\end{enumerate}

\begin{figure}[h!]
    \tikzset{every picture/.style={line width=0.75pt}} 

    \begin{tikzpicture}[x=0.75pt,y=0.75pt,yscale=-1,xscale=1]
    
    \draw  [fill={rgb, 255:red, 155; green, 155; blue, 155 }  ,fill opacity=1 ] (101.35,129.25) -- (121.71,158) -- (81,158) -- cycle ;
    \draw  [fill={rgb, 255:red, 155; green, 155; blue, 155 }  ,fill opacity=1 ] (142.06,129.25) -- (162.42,158) -- (121.71,158) -- cycle ;
    \draw  [fill={rgb, 255:red, 155; green, 155; blue, 155 }  ,fill opacity=1 ] (121.71,100.5) -- (142.06,129.25) -- (101.35,129.25) -- cycle ;
    \draw  [fill={rgb, 255:red, 155; green, 155; blue, 155 }  ,fill opacity=1 ] (182.77,129.25) -- (203.12,158) -- (162.42,158) -- cycle ;
    \draw  [fill={rgb, 255:red, 155; green, 155; blue, 155 }  ,fill opacity=1 ] (162.42,100.5) -- (182.77,129.25) -- (142.06,129.25) -- cycle ;
    \draw  [fill={rgb, 255:red, 155; green, 155; blue, 155 }  ,fill opacity=1 ] (142.06,71.75) -- (162.42,100.5) -- (121.71,100.5) -- cycle ;
    \draw  [fill={rgb, 255:red, 155; green, 155; blue, 155 }  ,fill opacity=1 ] (223.48,129.25) -- (243.83,158) -- (203.12,158) -- cycle ;
    \draw  [fill={rgb, 255:red, 155; green, 155; blue, 155 }  ,fill opacity=1 ] (203.12,100.5) -- (223.48,129.25) -- (182.77,129.25) -- cycle ;
    \draw  [fill={rgb, 255:red, 155; green, 155; blue, 155 }  ,fill opacity=1 ] (182.77,71.75) -- (203.12,100.5) -- (162.42,100.5) -- cycle ;
    \draw  [fill={rgb, 255:red, 155; green, 155; blue, 155 }  ,fill opacity=1 ] (162.42,43) -- (182.77,71.75) -- (142.06,71.75) -- cycle ;

    \end{tikzpicture}
    \caption{The $4$-fold half-hollow edgewise subdivision of a $2$-simplex}
    \end{figure}
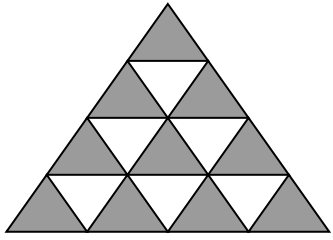

\begin{example}[The $\hesd$ construction for graphs]\label{ex:graph}
    Let $r > 0$ and $G$ be a graph with vertex set $[n]$ and edge set $E$. Condition~\ref{hesd1} implies edges of $\hesd(G, r)$ are of the form $\ba = a_i \be_i + a_j \be_j, \bb = b_i \be_i + b_j \be_j$ for $\{i, j\} \in E$. Condition~\ref{hesd2} then implies $\hesd(G, r)$ is obtained from $G$ by replacing every edge $\{i,j\}$ of $G$ by a path on $r + 1$ vertices with end points $i, j$. 
\end{example}
 
The collapsible property can be understood in terms of incidence complexes as follows.

\begin{proposition}\label{t:collapsiblespread}
    Let $\Delta$ be a $d$-dimensional pure collapsible simplicial complex and $r > 0$. Then $\F(\hesd(\Delta(d), r))$ has maximal analytic spread.
\end{proposition}

\begin{proof}
    By~\cref{l:analyticspreadrank} and since $\hesd(\Delta(d), r)$ is pure, the analytic spread of $\F(\hesd(\Delta(d), r))$ is equal to the rank of the log matrix $M$ of $\F(\hesd(\Delta(d), r))$. We will show that under the assumptions of the theorem, this matrix is upper triangular with nonzero diagonal entries, and hence must have full rank.
  
    We begin by showing that for a simplex $\Gamma$ on $n$ vertices, the log matrix $N$ of the facet ideal of the complex $\hesd(\Gamma, r) = \tuple{F_1, \dots, F_s}$ is upper triangular. In order to show this, we need to find an injection $\psi$ from the facets of $\hesd(\Gamma, r)$ to the vertices of $\hesd(\Gamma, r)$ such that (after reordering the facets of $\hesd(\Gamma, r)$) we have $\psi(F_i) = v_i \not \in F_j$ for $j > i$. Since by~\ref{hesd2} the set of facets of $\hesd(\Gamma, r)$ is $\Omega_n^{r-1}$ and the set of vertices is $\Omega_n^{r}$, we may define our injection $\psi = \psi_1$ as 
    \begin{equation}\label{eq:psimap}
        \psi_1(\ba) = \ba + \be_1,        
    \end{equation}
    where the ordering of facets of $\hesd(\Gamma, r)$ is the lex order on the elements of $\Omega_n^{r-1}$, in other words, $\ba <_{\mbox{lex}} \bb$ if and only if there exists $t$ such that $a_i = b_i$ for all $i < t$ and $a_t < b_t$. Under these assumptions, $\psi(\ba)$ is a vertex of a facet corresponding to $\ba'$ if and only if $\ba' = \ba + \be_j$ for some $j$. As the facets are ordered lexicographically (as elements of $\Omega_n^{r-1}$), if $\ba + \be_1 = \bb + \be_j$ for some $j > 1$, we conclude $\ba <_{\mbox{lex}} \bb$, in other words, $\psi$ satisfies the desired property.
    
    Now for an arbitrary $d$-dimensional pure collapsible complex $\Delta$, let $\sigma_1, \dots, \sigma_s$ be a sequence of $(d-1)$--faces such that after performing elementary collapses using the $\sigma_i$, the resulting complex is a $(d-1)$-dimensional complex. To see that the log matrix $M$ of $\F(\hesd(\Delta(d), r))$ is upper triangular (after reordering rows and columns), for a facet $T$ of $\Delta(d)$, let 
    $$
        C_T = \{\tau \st \cup_{\ba \in \tau}\supp(\ba) = T  \mbox{ and } \tau \mbox{ a facet of $\hesd(\Delta(d), r)$}\}.
    $$
    Then the previous arguments show that for $T_1$, the unique facet of $\Delta$ containing $\sigma_1$, there exists an injection $\psi_1$ from facets in $C_{T_1}$, to vertices of $\hesd(\Delta(d), r)$, where we reorder the vertices to guarantee that $\be_1$ is the coordinate corresponding to $\sigma_1$. After deleting every face of $\hesd(\Delta(d), r)$ that is contained only in facets of $C_{T_1}$, we obtain the complex $\hesd(\Delta'(d), r)$, where $\Delta'$ is the complex obtained after the first elementary collapse involving $\sigma_1$. We may now repeat the argument for the unique facet $T_2$ of $\Delta'$ that contains $\sigma_2$. Inductively, since the ridges $\sigma_1, \dots, \sigma_s$ were chosen such that the sequence of elementary collapses using them yields a $(d-1)$--dimensional complex, we conclude the matrix is upper triangular.     
\end{proof}

\begin{figure}[h]

    \tikzset{every picture/.style={line width=0.75pt}} 

    \begin{tikzpicture}[x=0.75pt,y=0.75pt,yscale=-1,xscale=1]
    
    \draw  [fill={rgb, 255:red, 155; green, 155; blue, 155 }  ,fill opacity=1 ] (171,131.05) -- (136.06,91) -- (206.06,91.11) -- cycle ;
    \draw  [fill={rgb, 255:red, 155; green, 155; blue, 155 }  ,fill opacity=1 ] (101,130.95) -- (66.06,90.89) -- (136.06,91) -- cycle ;
    \draw  [fill={rgb, 255:red, 155; green, 155; blue, 155 }  ,fill opacity=1 ] (135.94,171) -- (101,130.95) -- (171,131.05) -- cycle ;
    \draw  [fill={rgb, 255:red, 155; green, 155; blue, 155 }  ,fill opacity=1 ] (135.9,91.2) -- (171,131.05) -- (101.11,131.34) -- cycle ;
    \draw  [fill={rgb, 255:red, 155; green, 155; blue, 155 }  ,fill opacity=1 ] (249.21,56.5) -- (266.42,80) -- (232,80) -- cycle ;
    \draw  [fill={rgb, 255:red, 155; green, 155; blue, 155 }  ,fill opacity=1 ] (283.62,56.5) -- (300.83,80) -- (266.42,80) -- cycle ;
    \draw  [fill={rgb, 255:red, 155; green, 155; blue, 155 }  ,fill opacity=1 ] (266.42,33) -- (283.62,56.5) -- (249.21,56.5) -- cycle ;
    \draw  [fill={rgb, 255:red, 155; green, 155; blue, 155 }  ,fill opacity=1 ] (353.62,103.35) -- (336.21,80) -- (370.62,79.69) -- cycle ;
    \draw  [fill={rgb, 255:red, 155; green, 155; blue, 155 }  ,fill opacity=1 ] (319.21,103.65) -- (301.79,80.31) -- (336.21,80) -- cycle ;
    \draw  [fill={rgb, 255:red, 155; green, 155; blue, 155 }  ,fill opacity=1 ] (336.63,127) -- (319.21,103.65) -- (353.62,103.35) -- cycle ;
    \draw  [fill={rgb, 255:red, 155; green, 155; blue, 155 }  ,fill opacity=1 ] (388.21,55.5) -- (405.42,79) -- (371,79) -- cycle ;
    \draw  [fill={rgb, 255:red, 155; green, 155; blue, 155 }  ,fill opacity=1 ] (422.62,55.5) -- (439.83,79) -- (405.42,79) -- cycle ;
    \draw  [fill={rgb, 255:red, 155; green, 155; blue, 155 }  ,fill opacity=1 ] (405.42,32) -- (422.62,55.5) -- (388.21,55.5) -- cycle ;
    \draw  [fill={rgb, 255:red, 155; green, 155; blue, 155 }  ,fill opacity=1 ] (319.21,150.5) -- (336.42,174) -- (302,174) -- cycle ;
    \draw  [fill={rgb, 255:red, 155; green, 155; blue, 155 }  ,fill opacity=1 ] (353.62,150.5) -- (370.83,174) -- (336.42,174) -- cycle ;
    \draw  [fill={rgb, 255:red, 155; green, 155; blue, 155 }  ,fill opacity=1 ] (336.42,127) -- (353.62,150.5) -- (319.21,150.5) -- cycle ;
    \draw    (266.42,33) -- (266.83,50) ;
    \draw [shift={(266.83,50)}, rotate = 88.6] [color={rgb, 255:red, 0; green, 0; blue, 0 }  ][fill={rgb, 255:red, 0; green, 0; blue, 0 }  ][line width=0.75]      (0, 0) circle [x radius= 3.35, y radius= 3.35]   ;
    \draw [shift={(266.42,33)}, rotate = 88.6] [color={rgb, 255:red, 0; green, 0; blue, 0 }  ][fill={rgb, 255:red, 0; green, 0; blue, 0 }  ][line width=0.75]      (0, 0) circle [x radius= 3.35, y radius= 3.35]   ;
    \draw    (283.62,56.5) -- (284.04,73.5) ;
    \draw [shift={(284.04,73.5)}, rotate = 88.6] [color={rgb, 255:red, 0; green, 0; blue, 0 }  ][fill={rgb, 255:red, 0; green, 0; blue, 0 }  ][line width=0.75]      (0, 0) circle [x radius= 3.35, y radius= 3.35]   ;
    \draw [shift={(283.62,56.5)}, rotate = 88.6] [color={rgb, 255:red, 0; green, 0; blue, 0 }  ][fill={rgb, 255:red, 0; green, 0; blue, 0 }  ][line width=0.75]      (0, 0) circle [x radius= 3.35, y radius= 3.35]   ;
    \draw    (232,80) -- (249.62,73.5) ;
    \draw [shift={(249.62,73.5)}, rotate = 339.76] [color={rgb, 255:red, 0; green, 0; blue, 0 }  ][fill={rgb, 255:red, 0; green, 0; blue, 0 }  ][line width=0.75]      (0, 0) circle [x radius= 3.35, y radius= 3.35]   ;
    \draw [shift={(232,80)}, rotate = 339.76] [color={rgb, 255:red, 0; green, 0; blue, 0 }  ][fill={rgb, 255:red, 0; green, 0; blue, 0 }  ][line width=0.75]      (0, 0) circle [x radius= 3.35, y radius= 3.35]   ;
    \draw    (321.83,87) -- (301.79,80.31) ;
    \draw [shift={(301.79,80.31)}, rotate = 198.47] [color={rgb, 255:red, 0; green, 0; blue, 0 }  ][fill={rgb, 255:red, 0; green, 0; blue, 0 }  ][line width=0.75]      (0, 0) circle [x radius= 3.35, y radius= 3.35]   ;
    \draw [shift={(321.83,87)}, rotate = 198.47] [color={rgb, 255:red, 0; green, 0; blue, 0 }  ][fill={rgb, 255:red, 0; green, 0; blue, 0 }  ][line width=0.75]      (0, 0) circle [x radius= 3.35, y radius= 3.35]   ;
    \draw    (336.21,80) -- (355.83,87) ;
    \draw [shift={(355.83,87)}, rotate = 19.63] [color={rgb, 255:red, 0; green, 0; blue, 0 }  ][fill={rgb, 255:red, 0; green, 0; blue, 0 }  ][line width=0.75]      (0, 0) circle [x radius= 3.35, y radius= 3.35]   ;
    \draw [shift={(336.21,80)}, rotate = 19.63] [color={rgb, 255:red, 0; green, 0; blue, 0 }  ][fill={rgb, 255:red, 0; green, 0; blue, 0 }  ][line width=0.75]      (0, 0) circle [x radius= 3.35, y radius= 3.35]   ;
    \draw    (340.83,110) -- (319.21,103.65) ;
    \draw [shift={(319.21,103.65)}, rotate = 196.36] [color={rgb, 255:red, 0; green, 0; blue, 0 }  ][fill={rgb, 255:red, 0; green, 0; blue, 0 }  ][line width=0.75]      (0, 0) circle [x radius= 3.35, y radius= 3.35]   ;
    \draw [shift={(340.83,110)}, rotate = 196.36] [color={rgb, 255:red, 0; green, 0; blue, 0 }  ][fill={rgb, 255:red, 0; green, 0; blue, 0 }  ][line width=0.75]      (0, 0) circle [x radius= 3.35, y radius= 3.35]   ;
    \draw    (371,79) -- (388.21,67.25) ;
    \draw [shift={(388.21,67.25)}, rotate = 325.67] [color={rgb, 255:red, 0; green, 0; blue, 0 }  ][fill={rgb, 255:red, 0; green, 0; blue, 0 }  ][line width=0.75]      (0, 0) circle [x radius= 3.35, y radius= 3.35]   ;
    \draw [shift={(371,79)}, rotate = 325.67] [color={rgb, 255:red, 0; green, 0; blue, 0 }  ][fill={rgb, 255:red, 0; green, 0; blue, 0 }  ][line width=0.75]      (0, 0) circle [x radius= 3.35, y radius= 3.35]   ;
    \draw    (388.21,55.5) -- (405.42,43.75) ;
    \draw [shift={(405.42,43.75)}, rotate = 325.67] [color={rgb, 255:red, 0; green, 0; blue, 0 }  ][fill={rgb, 255:red, 0; green, 0; blue, 0 }  ][line width=0.75]      (0, 0) circle [x radius= 3.35, y radius= 3.35]   ;
    \draw [shift={(388.21,55.5)}, rotate = 325.67] [color={rgb, 255:red, 0; green, 0; blue, 0 }  ][fill={rgb, 255:red, 0; green, 0; blue, 0 }  ][line width=0.75]      (0, 0) circle [x radius= 3.35, y radius= 3.35]   ;
    \draw    (405.42,79) -- (422.62,67.25) ;
    \draw [shift={(422.62,67.25)}, rotate = 325.67] [color={rgb, 255:red, 0; green, 0; blue, 0 }  ][fill={rgb, 255:red, 0; green, 0; blue, 0 }  ][line width=0.75]      (0, 0) circle [x radius= 3.35, y radius= 3.35]   ;
    \draw [shift={(405.42,79)}, rotate = 325.67] [color={rgb, 255:red, 0; green, 0; blue, 0 }  ][fill={rgb, 255:red, 0; green, 0; blue, 0 }  ][line width=0.75]      (0, 0) circle [x radius= 3.35, y radius= 3.35]   ;
    \draw    (319.21,150.5) -- (319.62,167.5) ;
    \draw [shift={(319.62,167.5)}, rotate = 88.6] [color={rgb, 255:red, 0; green, 0; blue, 0 }  ][fill={rgb, 255:red, 0; green, 0; blue, 0 }  ][line width=0.75]      (0, 0) circle [x radius= 3.35, y radius= 3.35]   ;
    \draw [shift={(319.21,150.5)}, rotate = 88.6] [color={rgb, 255:red, 0; green, 0; blue, 0 }  ][fill={rgb, 255:red, 0; green, 0; blue, 0 }  ][line width=0.75]      (0, 0) circle [x radius= 3.35, y radius= 3.35]   ;
    \draw    (353.62,150.5) -- (354.04,167.5) ;
    \draw [shift={(354.04,167.5)}, rotate = 88.6] [color={rgb, 255:red, 0; green, 0; blue, 0 }  ][fill={rgb, 255:red, 0; green, 0; blue, 0 }  ][line width=0.75]      (0, 0) circle [x radius= 3.35, y radius= 3.35]   ;
    \draw [shift={(353.62,150.5)}, rotate = 88.6] [color={rgb, 255:red, 0; green, 0; blue, 0 }  ][fill={rgb, 255:red, 0; green, 0; blue, 0 }  ][line width=0.75]      (0, 0) circle [x radius= 3.35, y radius= 3.35]   ;
    \draw    (336.63,127) -- (337.04,144) ;
    \draw [shift={(337.04,144)}, rotate = 88.6] [color={rgb, 255:red, 0; green, 0; blue, 0 }  ][fill={rgb, 255:red, 0; green, 0; blue, 0 }  ][line width=0.75]      (0, 0) circle [x radius= 3.35, y radius= 3.35]   ;
    \draw [shift={(336.63,127)}, rotate = 88.6] [color={rgb, 255:red, 0; green, 0; blue, 0 }  ][fill={rgb, 255:red, 0; green, 0; blue, 0 }  ][line width=0.75]      (0, 0) circle [x radius= 3.35, y radius= 3.35]   ;
    
    \draw (131,191) node [anchor=north west][inner sep=0.75pt]   [align=left] {$\displaystyle \Delta $};
    \draw (310,189) node [anchor=north west][inner sep=0.75pt]   [align=left] {$\displaystyle \hesd(\Delta ( d) ,\ 2)$};
    \draw (261,11) node [anchor=north west][inner sep=0.75pt]   [align=left] {$\displaystyle 1$};
    \draw (219,84) node [anchor=north west][inner sep=0.75pt]   [align=left] {$\displaystyle 2$};
    \draw (287,38) node [anchor=north west][inner sep=0.75pt]   [align=left] {$\displaystyle 3$};
    \draw (304,58) node [anchor=north west][inner sep=0.75pt]   [align=left] {$\displaystyle 4$};
    \draw (338,57) node [anchor=north west][inner sep=0.75pt]   [align=left] {$\displaystyle 5$};
    \draw (306,106) node [anchor=north west][inner sep=0.75pt]   [align=left] {$\displaystyle 6$};
    \draw (364,57) node [anchor=north west][inner sep=0.75pt]   [align=left] {$\displaystyle 7$};
    \draw (377,33) node [anchor=north west][inner sep=0.75pt]   [align=left] {$\displaystyle 8$};
    \draw (399,83) node [anchor=north west][inner sep=0.75pt]   [align=left] {$\displaystyle 9$};
    \draw (344.83,118) node [anchor=north west][inner sep=0.75pt]   [align=left] {$\displaystyle 10$};
    \draw (298,132) node [anchor=north west][inner sep=0.75pt]   [align=left] {$\displaystyle 11$};
    \draw (362,144) node [anchor=north west][inner sep=0.75pt]   [align=left] {$\displaystyle 12$};

    \end{tikzpicture}
    \caption{The complex $\Delta$ on the left, and the complex $\hesd(\Delta(2), 2)$ on the right. The lines from vertices to facets correspond to the injection from~\cref{t:collapsiblespread} that shows the log matrix of $\F(\hesd(\Delta(2), 2))$ is upper triangular. Note that for every line, the vertex $i$ is contained only in facets corresponding to lower values of $i$, for example, the vertex $3$ is also contained in the facet containing the vertex $1$.}    
    \label{f:1}
\end{figure}
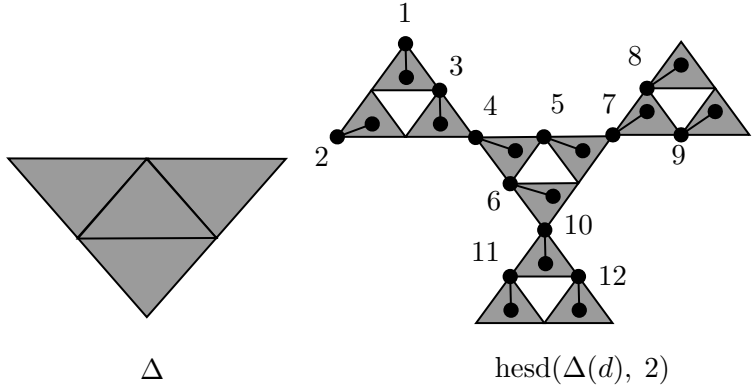

Our next goal is to show that given a $d$-dimensional complex $\Delta$, the log matrix of $\F(\hesd(\Delta(d), a - 1))$ represents multiplication maps in algebras of the form $A_\Delta(a) = \frac{\K[x_1,\dots, x_n]}{I_\Delta + (x_1^a, \dots, x_n^a)}$.

\begin{theorem}\label{p:highermatrices}
    Let $\Delta$ be a $d$-dimensional pure simplicial complex, $a > 1$ an integer, $L = x_1 + \dots + x_n \in R = \K[x_1,\dots, x_n]$ and $J = I_\Delta + (x_1^a, \dots, x_n^a)$. Then the matrix representing the map $\times L: (R/J)_{d(a-1)} \to (R/J)_{d(a-1) + 1}$ is the log matrix of $\F(\hesd(\Delta(d), a-1))$.
\end{theorem}

\begin{proof}
    Let $A_\Delta(a) = R/J$, $t = d(a-1)$ and $M$ be the matrix representing the map $\times L: A_\Delta(a)_{t} \to A_\Delta(a)_{t + 1}$. Note that nonzero monomials of $A_\Delta(a)_{t + 1}$ are of the form $m = x_{i_1}^{b_1}\dots x_{i_{d+1}}^{b_{d+1}}$, where $\{i_1, \dots, i_{d+1}\} \in \Delta$, $b_1 + \dots + b_{d+1} = t + 1$ and $b_j < a$ for all $j$. In particular, both conditions together imply $b_j > 0$ for every $j$. We then conclude $L\circ m$ is a sum of $d+1$ monomials of degree $t$ for every nonzero monomial $m \in A_\Delta(a)_{t+1}$. In other words, the row of $M$ corresponding to the monomial $m$ has $d+1$ nonzero entries, and they are all $1$. This implies $M$ is the log matrix of the facet ideal of a pure complex $\Gamma$ of dimension $d$.

    Our next goal is to show that $\Gamma = \hesd(\Delta(d), a - 1)$. Note that vertices of $\Gamma$ correspond to columns of $M$ and hence are labeled by nonzero monomials in $A_\Delta(a)_{t}$, and the facets of $\Gamma$ correspond to nonzero monomials in $A_\Delta(a)_{t + 1}$.

    In order to show the equality $\Gamma = \hesd(\Delta(d), a - 1)$, we first need to show a correspondence between vertices of $\Gamma$ and vertices of $\hesd(\Delta(d), a - 1)$ which by definition is $\Omega_{f_{d-1}}^{a-1} \subset \R^{f_{d-1}}$. For a facet $F$ of $\Delta$, set $\omega_F = \prod_{i \in F} x_i^{a-1}$. For every nonzero monomial $m \in A_\Delta(a)_t$, there exists a facet $F$ such that we can write $m$ as $m = \frac{\omega_F}{x_{j_1}^{c_1}\dots x_{j_s}^{c_s}}$, where $c_1 + \dots + c_s = a - 1$. Now, there is a bijection between ridges of $F$ and vertices of $F$ given by $v \mapsto F \setminus v$. Since ridges of $F$ are also ridges of $\Delta$, we define a function from the vertices of $\Gamma$ (which are monomials in $A_\Delta(a)_t$) to the vertices of $\hesd(\Delta(d), a - 1)$ given by 
    $$
        \varphi: V(\Gamma) \to \Omega_{f_{d - 1}}^{a-1}, \quad  \varphi\Big(\frac{\omega_F}{x_{j_1}^{c_1}\dots x_{j_s}^{c_s}}\Big) = c_1\be_{F \setminus j_1} + \dots + c_s \be_{F \setminus j_s},
    $$
    where $\be_{\tau}$ is the element in the canonical basis of $\R^{f_{d-1}}$ corresponding to the ridge $\tau$ of $\Delta$. We note that even though it is possible for a monomial $m$ to have two different representations $\frac{\omega_{F'}}{x_{j_1}^{c_1}\dots x_{j_s}^{c_s}} = \frac{\omega_{F''}}{x_{k_1}^{q_1}\dots x_{k_\ell}^{q_\ell}} = m \in A_\Delta(a)_t$, since the degree of the denominator is $a-1$ and the support of $m$ is a face of $\Delta$, this only happens if $m = \prod_{i \in \tau} x_i^{a-1}$ for some ridge $\tau = F' \setminus v' = F'' \setminus v''$ of $\Delta$. In this case, $\varphi(m) = (a-1)\be_{\tau}$ and so $\varphi$ is well defined. The map $\varphi$ is a surjection since for every vector $\ba = c_1 \be_{F \setminus j_1} + \dots + c_s \be_{F \setminus j_s} \in \Omega_{f_{d-1}}^{a-1}$, the monomial $m = \frac{\omega}{x_{j_1}^{c_1}\dots x_{j_s}^{c_s}} \in A_\Delta(a)_t$ is such that $\varphi(m) = \ba$.    
    To see that $\varphi$ is an injection, note that if $\varphi(m) = \varphi(m') = \ba = c_1 \be_{F \setminus j_1} + \dots + c_s \be_{F \setminus F_{j_s}}$, then we have $\supp(m) = \supp(m')$ and in particular $m = \frac{\omega_F}{x_{j_1}^{c_1}\dots x_{j_s}^{c_s}} = m'$.

    For a correspondence between facets of $\Gamma$ (which are monomials in $A_\Delta(a)_{t+1}$) and facets of $\hesd(\Delta(d), a - 1)$, note that every facet of $\Gamma$ is given by the set of monomials $\{\frac{m}{x_{i}} \st x_i \mid m \}$, where $m \in A_\Delta(a)_{t+1}$. Similarly to the case of monomials of degree $t$ described above, every monomial $m$ of degree $t + 1$ can be written as $m = \frac{\omega_F}{x_{j_1}^{c_1}\dots x_{j_s}^{c_s}}$ for some facet $F$ of $\Delta$, where $c_1 + \dots + c_s = a - 2$. Denoting by $\ba_m$ the vector associated to the monomial $m$ of degree $t + 1$ we see that facets of $\Gamma$ correspond to sets of the form $\{\ba_m + \be_{F \setminus i} \st i \in F\}$, and since $\ba_m$ is a vector with positive entries adding up to $a -2$, and $\supp(a)$ by definition is the set of ridges of a facet of $\Delta$, we conclude facets of $\Gamma$ correspond exactly to the sets of vectors satisfying~\ref{hesd1},~\ref{hesd2}.
    
    We have shown that there is a bijection between vertices of $\Gamma$ and $\hesd(\Delta(d), a - 1)$ that preserves facets, in particular both complexes have the same log matrices and so the result follows.
\end{proof}

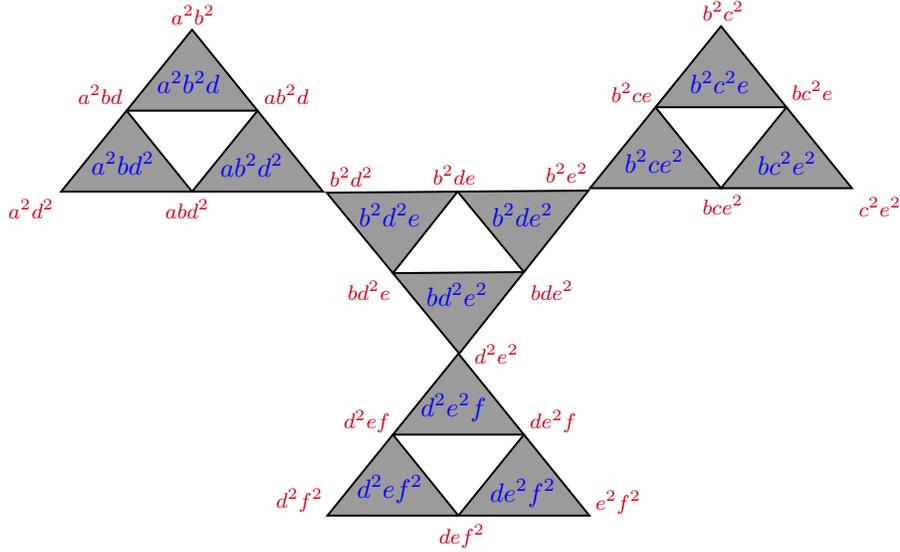
\begin{figure}
    \begin{center}

\tikzset{every picture/.style={line width=0.75pt}} 

\begin{tikzpicture}[x=0.75pt,y=0.75pt,yscale=-1,xscale=1]

\draw  [fill={rgb, 255:red, 155; green, 155; blue, 155 }  ,fill opacity=1 ] (244.02,69.62) -- (277.05,110.49) -- (211,110.49) -- cycle ;
\draw  [fill={rgb, 255:red, 155; green, 155; blue, 155 }  ,fill opacity=1 ] (310.07,69.62) -- (343.09,110.49) -- (277.05,110.49) -- cycle ;
\draw  [fill={rgb, 255:red, 155; green, 155; blue, 155 }  ,fill opacity=1 ] (277.05,28.74) -- (310.07,69.62) -- (244.02,69.62) -- cycle ;
\draw  [fill={rgb, 255:red, 155; green, 155; blue, 155 }  ,fill opacity=1 ] (444.4,151.11) -- (410.97,110.49) -- (477.01,109.95) -- cycle ;
\draw  [fill={rgb, 255:red, 155; green, 155; blue, 155 }  ,fill opacity=1 ] (378.36,151.64) -- (344.93,111.03) -- (410.97,110.49) -- cycle ;
\draw  [fill={rgb, 255:red, 155; green, 155; blue, 155 }  ,fill opacity=1 ] (411.78,192.25) -- (378.35,151.64) -- (444.39,151.09) -- cycle ;
\draw  [fill={rgb, 255:red, 155; green, 155; blue, 155 }  ,fill opacity=1 ] (510.76,67.88) -- (543.79,108.75) -- (477.74,108.75) -- cycle ;
\draw  [fill={rgb, 255:red, 155; green, 155; blue, 155 }  ,fill opacity=1 ] (576.81,67.88) -- (609.83,108.75) -- (543.79,108.75) -- cycle ;
\draw  [fill={rgb, 255:red, 155; green, 155; blue, 155 }  ,fill opacity=1 ] (543.79,27) -- (576.81,67.88) -- (510.76,67.88) -- cycle ;
\draw  [fill={rgb, 255:red, 155; green, 155; blue, 155 }  ,fill opacity=1 ] (378.35,233.12) -- (411.38,274) -- (345.33,274) -- cycle ;
\draw  [fill={rgb, 255:red, 155; green, 155; blue, 155 }  ,fill opacity=1 ] (444.4,233.12) -- (477.42,274) -- (411.38,274) -- cycle ;
\draw  [fill={rgb, 255:red, 155; green, 155; blue, 155 }  ,fill opacity=1 ] (411.38,192.25) -- (444.4,233.12) -- (378.35,233.12) -- cycle ;

\draw (257.8,47.99) node [anchor=north west][inner sep=0.75pt]  [font=\small,color={rgb, 255:red, 0; green, 0; blue, 255 }  ,opacity=1 ] [align=left] {$\displaystyle a^{2} b^{2} d$};
\draw (224.8,88.99) node [anchor=north west][inner sep=0.75pt]  [font=\small,color={rgb, 255:red, 0; green, 0; blue, 255 }  ,opacity=1 ] [align=left] {$\displaystyle a^{2} bd^{2}$};
\draw (289.8,89.99) node [anchor=north west][inner sep=0.75pt]  [font=\small,color={rgb, 255:red, 0; green, 0; blue, 255 }  ,opacity=1 ] [align=left] {$\displaystyle ab^{2} d^{2}$};
\draw (359.8,114.99) node [anchor=north west][inner sep=0.75pt]  [font=\small,color={rgb, 255:red, 0; green, 0; blue, 255 }  ,opacity=1 ] [align=left] {$\displaystyle b^{2} d^{2} e$};
\draw (426.8,114.99) node [anchor=north west][inner sep=0.75pt]  [font=\small,color={rgb, 255:red, 0; green, 0; blue, 255 }  ,opacity=1 ] [align=left] {$\displaystyle b^{2} de^{2}$};
\draw (393.8,154.99) node [anchor=north west][inner sep=0.75pt]  [font=\small,color={rgb, 255:red, 0; green, 0; blue, 255 }  ,opacity=1 ] [align=left] {$\displaystyle bd^{2} e^{2}$};
\draw (493.8,87.99) node [anchor=north west][inner sep=0.75pt]  [font=\small,color={rgb, 255:red, 0; green, 0; blue, 255 }  ,opacity=1 ] [align=left] {$\displaystyle b^{2} ce^{2}$};
\draw (526.8,47.99) node [anchor=north west][inner sep=0.75pt]  [font=\small,color={rgb, 255:red, 0; green, 0; blue, 255 }  ,opacity=1 ] [align=left] {$\displaystyle b^{2} c^{2} e$};
\draw (560.8,88.99) node [anchor=north west][inner sep=0.75pt]  [font=\small,color={rgb, 255:red, 0; green, 0; blue, 255 }  ,opacity=1 ] [align=left] {$\displaystyle bc^{2} e^{2}$};
\draw (390.8,210.99) node [anchor=north west][inner sep=0.75pt]  [font=\small,color={rgb, 255:red, 0; green, 0; blue, 255 }  ,opacity=1 ] [align=left] {$\displaystyle d^{2} e^{2} f$};
\draw (358.8,251.99) node [anchor=north west][inner sep=0.75pt]  [font=\small,color={rgb, 255:red, 0; green, 0; blue, 255 }  ,opacity=1 ] [align=left] {$\displaystyle d^{2} ef^{2}$};
\draw (425.8,253.99) node [anchor=north west][inner sep=0.75pt]  [font=\small,color={rgb, 255:red, 0; green, 0; blue, 255 }  ,opacity=1 ] [align=left] {$\displaystyle de^{2} f^{2}$};
\draw (265,15) node [anchor=north west][inner sep=0.75pt]  [font=\scriptsize,color={rgb, 255:red, 208; green, 2; blue, 27 }  ,opacity=1 ] [align=left] {$\displaystyle a^{2} b^{2}$};
\draw (183,113) node [anchor=north west][inner sep=0.75pt]  [font=\scriptsize,color={rgb, 255:red, 208; green, 2; blue, 27 }  ,opacity=1 ] [align=left] {$\displaystyle a^{2} d^{2}$};
\draw (345,97) node [anchor=north west][inner sep=0.75pt]  [font=\scriptsize,color={rgb, 255:red, 208; green, 2; blue, 27 }  ,opacity=1 ] [align=left] {$\displaystyle b^{2} d^{2}$};
\draw (218,55) node [anchor=north west][inner sep=0.75pt]  [font=\scriptsize,color={rgb, 255:red, 208; green, 2; blue, 27 }  ,opacity=1 ] [align=left] {$\displaystyle a^{2} bd$};
\draw (312,55) node [anchor=north west][inner sep=0.75pt]  [font=\scriptsize,color={rgb, 255:red, 208; green, 2; blue, 27 }  ,opacity=1 ] [align=left] {$\displaystyle ab^{2} d$};
\draw (262,113) node [anchor=north west][inner sep=0.75pt]  [font=\scriptsize,color={rgb, 255:red, 208; green, 2; blue, 27 }  ,opacity=1 ] [align=left] {$\displaystyle abd^{2}$};
\draw (397,96) node [anchor=north west][inner sep=0.75pt]  [font=\scriptsize,color={rgb, 255:red, 208; green, 2; blue, 27 }  ,opacity=1 ] [align=left] {$\displaystyle b^{2} de$};
\draw (446.39,154.09) node [anchor=north west][inner sep=0.75pt]  [font=\scriptsize,color={rgb, 255:red, 208; green, 2; blue, 27 }  ,opacity=1 ] [align=left] {$\displaystyle bde^{2}$};
\draw (354,154) node [anchor=north west][inner sep=0.75pt]  [font=\scriptsize,color={rgb, 255:red, 208; green, 2; blue, 27 }  ,opacity=1 ] [align=left] {$\displaystyle bd^{2} e$};
\draw (352,219) node [anchor=north west][inner sep=0.75pt]  [font=\scriptsize,color={rgb, 255:red, 208; green, 2; blue, 27 }  ,opacity=1 ] [align=left] {$\displaystyle d^{2} ef$};
\draw (446,219) node [anchor=north west][inner sep=0.75pt]  [font=\scriptsize,color={rgb, 255:red, 208; green, 2; blue, 27 }  ,opacity=1 ] [align=left] {$\displaystyle de^{2} f$};
\draw (400,277) node [anchor=north west][inner sep=0.75pt]  [font=\scriptsize,color={rgb, 255:red, 208; green, 2; blue, 27 }  ,opacity=1 ] [align=left] {$\displaystyle def^{2}$};
\draw (487,54) node [anchor=north west][inner sep=0.75pt]  [font=\scriptsize,color={rgb, 255:red, 208; green, 2; blue, 27 }  ,opacity=1 ] [align=left] {$\displaystyle b^{2} ce$};
\draw (578,53) node [anchor=north west][inner sep=0.75pt]  [font=\scriptsize,color={rgb, 255:red, 208; green, 2; blue, 27 }  ,opacity=1 ] [align=left] {$\displaystyle bc^{2} e$};
\draw (533,111) node [anchor=north west][inner sep=0.75pt]  [font=\scriptsize,color={rgb, 255:red, 208; green, 2; blue, 27 }  ,opacity=1 ] [align=left] {$\displaystyle bce^{2}$};
\draw (533,13) node [anchor=north west][inner sep=0.75pt]  [font=\scriptsize,color={rgb, 255:red, 208; green, 2; blue, 27 }  ,opacity=1 ] [align=left] {$\displaystyle b^{2} c^{2}$};
\draw (611.83,111.75) node [anchor=north west][inner sep=0.75pt]  [font=\scriptsize,color={rgb, 255:red, 208; green, 2; blue, 27 }  ,opacity=1 ] [align=left] {$\displaystyle c^{2} e^{2}$};
\draw (454,95) node [anchor=north west][inner sep=0.75pt]  [font=\scriptsize,color={rgb, 255:red, 208; green, 2; blue, 27 }  ,opacity=1 ] [align=left] {$\displaystyle b^{2} e^{2}$};
\draw (418,186) node [anchor=north west][inner sep=0.75pt]  [font=\scriptsize,color={rgb, 255:red, 208; green, 2; blue, 27 }  ,opacity=1 ] [align=left] {$\displaystyle d^{2} e^{2}$};
\draw (479,260) node [anchor=north west][inner sep=0.75pt]  [font=\scriptsize,color={rgb, 255:red, 208; green, 2; blue, 27 }  ,opacity=1 ] [align=left] {$\displaystyle e^{2} f^{2}$};
\draw (318,259) node [anchor=north west][inner sep=0.75pt]  [font=\scriptsize,color={rgb, 255:red, 208; green, 2; blue, 27 }  ,opacity=1 ] [align=left] {$\displaystyle d^{2} f^{2}$};

\end{tikzpicture}
    \end{center}
    \caption{A labeling of the vertices and facets of $\hesd(\Delta(2), 2)$, where $\Delta$ is the complex from~\cref{f:1} and $\F(\Delta) = (abd, bde, bce, def)$, that shows the log-matrix of $\F(\hesd(\Delta(2), 2))$ is the multiplication map $\times L: A_\Delta(3)_{4} \to A_\Delta(3)_5$. Note that applying $\times L^T$ to the monomial label of a triangle yields the sum of the labels of its vertices. As an example, the triangle labeled by $a^2 b^2 d$ corresponds to the vector $\ba = \be_{\{a,b\}}$ and its vertices labeled by $a^2 b^2$, $a^2 b d$ and $a b^2 d$ correspond to the vectors $2\be_{\{a,b\}}$, $\be_{\{a,b\}} + \be_{\{a,d\}}$ and $\be_{\{a,b\}} + \be_{\{b,d\}}$.}
\end{figure}
\cref{t:collapsiblespread,p:highermatrices} imply the following.

\begin{corollary}[\emph{A setting where unexpected sops do not exist}]\label{c:nonexistence}
    Let $\Delta$ be a $d$-dimensional Cohen-Macaulay collapsible complex. Then for a fixed $a > 0$, there does not exist an $a$-unexpected system of parameters of $\Delta$ of total degree $t > d(a-1)$.
\end{corollary}

\begin{proof}
    Let $J = I_\Delta + (x_1^a, \dots, x_n^a), \subset R = \K[x_1,\dots, x_n]$ and $L = x_1 + \dots + x_n$. By~\cref{p:highermatrices,l:analyticspreadrank} we have 
    $$
        \rk \times L: (R/J)_{d(a-1)} \to (R/J)_{d(a-1) + 1} = \ell(\F(\hesd(\Delta(d), a - 1))).
    $$
    \cref{t:collapsiblespread} then implies the map is surjective.
 
    If $\theta_1, \dots, \theta_{d+1}$ is an $a$-unexpected sop of $\Delta$ of total degree $t$, the socle degree of the algebra $\frac{R}{I_\Delta + (\theta_1, \dots, \theta_{d+1})}$ is $t$ and so by~\cref{p:unexpectedfailure} a top degree generator $F$ in the inverse system of $I_\Delta + (\theta_1, \dots, \theta_{d+1})$ corresponds to a nonzero element in the kernel of $\times L^T: (R/J)_{t} \to (R/J)_{t - 1}$, which in turn implies the map $\times L: (R/J)_{t-1} \to (R/J)_{t}$ is not surjective. This is a contradiction since once a map $\times L: (R/J)_{i} \to (R/J)_{i+1}$ is surjective, every multiplication map $\times L: (R/J)_j \to (R/J)_{j+1}$ for $j \geq i$ is also surjective (see for example~\cite[Proposition 2.1]{MMN2011}).
\end{proof}

\begin{remark}
    We note that even though~\cref{c:nonexistence} can be stated (and proved) without the notion of analytic spread, the advantage of emphasizing this connection is that certain relationships between Lefschetz properties and other topics in commutative algebra such as symbolic powers and ideals of linear type become more apparent. Such connections are made explicit for example in~\cite{H2025}.
\end{remark}

\cref{c:nonexistence} does not guarantee that a collapsible complex does not have unexpected sops, but we do not know of a collapsible complex that has one. 

The following example is motivated by~\cite[Example 3.7]{HN2025}. We do not include the definition of shellable complexes here, since the only property we need is the well-known fact that a contractible shellable complex is collapsible (see for example~\cite[Section 1.1.3]{BZ2011} and~\cite[Section 11]{B1995}).

\begin{example}[\emph{Can a collapsible complex have an unexpected sop?}]
    Let 
    $$
        I = (x_1 x_4, x_1 x_5, x_2 x_4, x_2 x_5, x_3 x_4, x_3 x_5, x_4 x_5) + (x_1 y_1, \dots, x_5 y_5) \in \K[x_1,\dots, x_5, y_1, \dots, y_5]
    $$
    and $\Delta$ the Stanley-Reisner complex of $I$. It is known that $\Delta$ is a shellable~\cite[Theorem 1.5]{DE2009} simplicial ball~\cite[Section 3]{M2011}, and hence it must also be collapsible. Macaulay2~\cite{M2} computations show that the Hilbert functions of $J = I + (x_1^4,\dots, x_5^4, y_1^4, \dots, y_5^4)$ and $J + L$, where $L = x_1 + \dots + x_5 + y_1 + \dots + y_5$ attain the following values:
    $$
    \begin{array}{c|cccccccccccc} 
        i & 0 & 1 & 2 &3 &4 & 5 & 6 & 7 &8 & 9 & 10 & \\
        \hline
        \HF(i, \frac{R}{J}) &1&10 & 43& 126& 285& 520& 793& 1026& 1134 & 1076 & 870 \\
        \HF(i, \frac{R}{J + L}) & 1 & 9 & 33& 83& 159& 235& 273& 233& 108 & 7 & 0
        \end{array}
    $$
    In particular, we conclude $\dim K = 7$, where $K = \ker \times L^T: (R/J)_{9} \to (R/J)_{8}$. Note that~\cref{p:highermatrices,t:collapsiblespread} imply the map $\times L: (R/J)_{12} \to (R/J)_{13}$ is surjective.

    Let $F_1, \dots, F_7$ be a basis of $K$. If $\Delta$ has an $4$-unexpected sop $\theta_1, \dots, \theta_{5}$, it must have total degree $9$. Since $\Delta$ is not a $\K$-homology sphere, previous arguments involving the unique Macaulay dual generator do not apply. Instead, the only information we have is that one of the generators of the inverse system of $I + (\theta_1, \dots, \theta_5)$ (if it exists) is a linear combination of $F_1, \dots, F_7$. The other generators could have lower degree, and hence a priori do not imply failure of Lefschetz, as they are in the kernel of a map that can be surjective (the values of the Hilbert function are increasing). 
\end{example}

\begin{remark}
    Throughout this section we dealt with collapsible complexes. Every result in this section holds if we assume that $\Delta$ is a $d$-dimensional complex that collapses down to a $(d-1)$--dimensional complex.
\end{remark}

Finally, we note that the questions we consider about WLP of monomial algebras can be viewed as more general cases of questions about Gorenstein ideals and of polynomials $F$ such that 
$$
    \nabla \cdot  F = \frac{\partial F}{\partial x_1} + \dots + \frac{\partial F}{\partial x_n} = 0.
$$
Below we state one version of the results implied by our perspective in these special cases.
\begin{corollary}\label{c:basecase}
    Let $\K$ be a field of characteristic $0$, $a > 0$ an integer, $F \in R = \K[x_1,\dots,x_n]$ be a polynomial such that $\nabla \cdot F = 0$ and $I \subset R$ an artinian Gorenstein ideal. 
    \begin{enumerate}
        \item If $F$ is a sum of monomials not divisible by $x_i^a$ for any $i$, then $\deg F \leq \frac{n(a - 1)}{2}$.
        \item If $x_i^a \in I$ for every $i$ and $L = x_1 + \dots + x_n \in I$, the socle degree of $I$ is at most $\frac{n(a-1)}{2}$.
    \end{enumerate}
\end{corollary}

\begin{proof}
    First note that it is known by a result of Stanley~\cite{S80} that monomial complete intersections satisfy the SLP.
    \begin{enumerate}
        \item Assume $\deg F > \frac{n(a - 1)}{2}$. Since $\nabla \cdot F = L \bullet F = 0$, we conclude there exists an ideal $I$ such that $I^{\perp} = (F)$, and by the assumptions on $F$ we have $L \in I$ and $x_i^a \in I$ for every $i$. Let $G$ be the polynomial such that $I^{-1} = (G)$ and $A = R/(x_1^a, \dots, x_n^a)$. By~\cref{l:contraction} we conclude $G$ is in the kernel of the map 
        $$
        \times L^T: A_{\deg F} \to A_{\deg F - 1}.
        $$
        The assumption on the degree of $F$ then gives us a contradiction since the socle degree of $A$ is $n(a-1)$ and $A$ satisfies the SLP.
        \item By Macaulay duality a similar argument as above works for the second statement.
    \end{enumerate}
\end{proof}

From the perspective of~\cref{c:basecase}, our questions have the restriction that the monomials with nonzero coefficient in $F$ are supported in a face of a simplicial complex $\Delta$.

\section{An application: WLP of $1$-dimensional complexes}\label{s:1dim}

In this section we exemplify how the results and ideas from~\cref{s:collapsible} can be used in order to determine the WLP of algebras of the form $A_\Delta(a) = \frac{R}{I_\Delta + (x_1^a, \dots, x_n^a)}$. Here we characterize the WLP of such algebras when $\Delta$ is a $1$-dimensional Cohen-Macaulay complex (a connected graph), and the underlying field $\K$ has characteristic zero.

Note that since $\Delta$ is a connected graph, its $f$-vector is given by $(1, v, e)$ where $v$ is the number of vertices of $\Delta$, and $e$ is the number of edges of $\Delta$. Moreover, there is a simple explicit formula for the Hilbert series of $A_\Delta(a)$.

\begin{lemma}\label{l:hilbert}
    Let $\Delta$ be a graph with $v$ vertices and $e$ edges, and $a > 0$. The Hilbert function of $A_{\Delta}(a)$ is 
    \begin{align*}
        \HF(t, A_\Delta(a)) = \begin{cases}
            1 \quad & t = 0 \\
            v + (t-1)e \quad & 0 < t < a \\
            e(2a - t - 1) \quad & a \leq t \leq 2a - 2 \\
            0 \quad & t > 2a - 2
        \end{cases}
    \end{align*}
    In particular, the socle degree of $A_G(a)$ is $2a - 2$, and 
    \begin{enumerate}
        \item If $v < e$, $\HF(a - 1, A_\Delta(a)) < \HF(a, A_\Delta(a))$
        \item If $v = e$, $\HF(a - 1, A_\Delta(a)) = \HF(a, A_\Delta(a))$
        \item If $v > e$, $\HF(a - 1, A_\Delta(a)) > \HF(a, A_\Delta(a))$
    \end{enumerate}
\end{lemma}

\begin{proof}
    Nonzero monomials in $A_\Delta(a)$ have one of the following two forms:
    $$
        x_i^{k} x_j^{l} \qfor 1 \leq k,l \leq a - 1 \qand x_i^c \qfor 0 \leq c \leq a - 1,
    $$ 
    A straightforward counting argument then implies the formulas for the Hilbert function. The last part of the statement follows directly since 
    $$
    \HF(a - 1, A_\Delta(a)) - \HF(a, A_\Delta(a)) = v + (a - 2)e - e(2a - a - 1) = v - e.
    $$
\end{proof}
 
Let $L = x_1 + \dots + x_n$. \cref{l:hilbert} implies the multiplication map $\times L: A_G(a)_{a-1} \to A_G(a)_a$ plays a key role in the WLP of $A_G(a)$ for an arbitrary connected graph $G$ and $a > 0$. In view of~\cref{ex:graph} we have the following variation of~\cref{p:highermatrices}.

\begin{proposition}\label{p:matrixsubdivision}
    The matrix that represents the multiplication map $\times L: A_\Delta(a)_{a} \to A_\Delta(a)_{a+1}$ is the incidence matrix of $f_1(\Delta)$ paths with $a-1$ vertices each. In particular, $\times L: A_\Delta(a)_a \to A_\Delta(a)_{a+1}$ is always surjective.
\end{proposition}

\begin{proof}
    Let $m$ be a nonzero monomial in $A_G(a)_{a}$. Then $m = x_i^{a-k} x_j^{k}$ for some edge $\{i,j\}$ of $G$ and $1 \leq k \leq a - 1$. In order to compute the matrix $M$ that represents the multiplication map $\times L: A_G(a)_{a} \to A_G(a)_{a + 1}$, we compute $\times L^T(m) = L \circ m$, for nonzero monomials of degree $a + 1$ in $A_G(a)$. For $m = x_i^{a + 1 - k}x_j^{k}$ with $2 \leq k \leq a - 1$, we have $L \circ m = x_i^{a - k} x_j^{k} + x_i^{a + 1 - k} x_j^{k - 1}$, and in particular, since $M$ is a $0,1$ matrix and every row of $M$ has exactly two nonzero entries, we conclude $M$ is the incidence matrix of a graph $H$.
 
    The vertex set of $H$ is the set \
    $$
        V(H) = \{x_i^{a - k} x_j^{k} \st \{i,j\} \in G \qand 1 \leq k \leq a - 1\},
    $$
    and the edges of $H$ are the pairs $\{x_i^{a - k} x_j^{k}, x_i^{a + 1 - k} x_j^{k - 1}\}$ where $\{i, j\}$ ranges over the edges of $G$ and $2 \leq k \leq a - 2$. The graph $H$ has a connected component for every $\{i, j\} \in G$, and inside each connected component there are $a - 1$ vertices and $a - 2$ edges. More specifically, $H$ is a forest where every connected component is a path.
    
    It is a well-known fact from graph theory that the rank of the incidence matrix of a graph $H$ is equal to the difference between the number of vertices of $H$ and the number of connected components of $H$~\cite[Theorem 8.2.1]{G2001}. This implies the matrix $M$ has full rank and the multiplication map is surjective.
\end{proof}

\cref{p:matrixsubdivision} guarantees a specific map is surjective. The following result due to Hausel~\cite{H2005} guarantees the injectivity of the first half of the maps.

\begin{theorem}[{\cite[Theorem 6.2]{H2005}}]\label{t:hausel}
    Let $R = \K[x_1,\dots, x_n]$ and $A = R/I$ be a monomial artinian level algebra of socle degree $t$, where the characteristic of $\K$ is $0$. Then for $L = x_1 + \dots + x_n$, the multiplication maps 
    $$
        \times L: A_i \to A_{i + 1}
    $$
    are injective for $0 \leq i \leq \lfloor \frac{t - 1}{2} \rfloor$.
\end{theorem}

We are now ready to prove the main result of this section.

\begin{theorem}[\emph{The WLP of $1$-dimensional complexes}]\label{t:wlp1dim}
    Let $\Delta$ be a connected graph with $f$-vector $(1, v, e)$ and $a > 1$ a fixed integer. Then $A_\Delta(a)$ has the WLP over a field of characteristic $0$ if and only if:
    \begin{enumerate}
        \item $v > e$ or
        \item $v \leq e$ and $\hesd(\Delta, a)$ is not bipartite.
    \end{enumerate} 
\end{theorem}

\begin{proof}
    First note that the socle degree of $A_\Delta(a)$ is $2a - 2$, and since $A_\Delta(a)$ is a level algebra for every $\Delta$ and $a$,~\cref{t:hausel} implies the multiplication maps $\times L: A_\Delta(a)_{i} \to A_\Delta(a)_{i + 1}$ are injective for every $1 \leq i \leq \lfloor \frac{2a - 3}{2} \rfloor = a - 2$. \cref{p:matrixsubdivision} implies the multiplication map $\times L : A_\Delta(a)_a \to A_\Delta(a)_{a + 1}$ is surjective, and hence by~\cite[Proposition 2.1]{MMN2011} every multiplication map $\times L: A_\Delta(a)_{i} \to A_\Delta(a)_{i + 1}$ is surjective for every $i \geq a$. In particular, the only map that can fail to have full rank is $\times L: A_G(a)_{a - 1} \to A_G(a)_{a}$. By~\cref{p:highermatrices},~\cref{ex:graph} and~\cite[Theorem 8.2.1]{G2001} this map has full rank if and only if either 
    \begin{enumerate}
       \item $v > e$ or
        \item $v \leq e$ and $\hesd(\Delta, a)$ is not bipartite.
    \end{enumerate}
\end{proof}

\begin{remark} 
    Note that if $\Delta$ is not connected, then the multiplication maps $\times L: A_\Delta(a)_{i} \to A_\Delta(a)_{i + 1}$ will be block matrices for every $i$, where blocks correspond to the multiplication maps of connected components. The WLP in this case then simply says that either every connected component is a tree (that is, the graph is a forest), or no connected component of $\hesd(\Delta, a)$ is bipartite. 
\end{remark}

\begin{remark}
    It turns out that the failure of WLP of algebras of the form $A_\Delta(a)$, where $\Delta$ is a connected graph is equivalent to the existence of an unexpected sop. This should be seen as the analogue of~\cite[Theorem 7.5]{CHMN2018}, where the authors show an equivalence between failure of the SLP in a specific degree, and the existence of unexpected curves. Here, instead of restricting the number of variables, we restrict the dimension of $\Delta$ in order to obtain the equivalence. A proof of this statement, as well as an explicit formula for the elements in the unexpected sop will appear in a future paper.
\end{remark}

\section{Further directions}\label{s:questions}

In this paper we introduced the concept of unexpected systems of parameters of simplicial complexes. We then explored settings where the existence (and nonexistence) of such sops can be guaranteed. We now state some of the natural questions that arise from our work.
 
One motivation for studying artinian reductions of Stanley-Reisner rings is to be able to prove inequalities of $f$ and $h$ vectors of special classes of simplicial complexes. Some of the well-known applications are the proofs of necessity for McMullen's $g$-conjecture for simplicial polytopes by Stanley~\cite{S80} and for simplicial spheres by Adiprasito~\cite{A2018}, Papadakis and Petrotou~\cite{PP2020} and Adiprasito, Papadakis and Petrotou~\cite{APP2021}.   
 
It is often the case (especially when there is no known underlying geometry), that proofs regarding Lefschetz properties of artinian reductions of face rings of simplicial spheres (and related classes) rely heavily on arguments involving general linear forms. This phenomenon leads to the interesting question (often mentioned in the literature, see for example~\cite[Open problem 7.4]{APP2021} and~\cite[Conjecture 1.3]{CJKMN2018}) of which systems of parameters of $I_\Delta$ define algebras satisfying the SLP. Systems of parameters that are combinatorially defined are natural candidates for such a question. 

The two currently known families of unexpected systems of parameters for large classes of homology spheres of arbitrary dimension are:

\begin{enumerate}
    \item The universal sop, defined by elementary symmetric polynomials~\cite{H2025B}, and 
    \item The colored sop for balanced homology spheres~(\cref{t:coloredunexpected})
\end{enumerate}

As is mentioned in~\cite[Section 8 and Question 8.1]{H2025B}, the universal sop is known to define algebras satisfying the SLP in very special cases. The SLP of colored sops of balanced spheres was already conjectured in~\cite[Conjecture 1.3]{CJKMN2018}. These remarks, together with computational evidence from Macaulay2~\cite{M2} lead us to the following question that generalizes~\cite[Question 8.1]{H2025B} and~\cite[Conjecture 1.3]{CJKMN2018}.
 
\begin{question}[\emph{Unexpected systems of parameters and the algebraic $g$-conjecture}]
    Let $\Delta$ be a $d$-dimensional $\Q$-homology sphere on vertex set $[n]$, and $\theta_1, \dots, \theta_{d+1}$ an $a$-unexpected sop of $\Delta$. Does the algebra 
    $$
        \frac{\C[x_1,\dots, x_n]}{I_\Delta + (\theta_1, \dots, \theta_{d+1})}
    $$
    satisfy the SLP?
\end{question}

As was pointed out in earlier sections, most of the arguments we use to guarantee the existence of unexpected sops for homology spheres rely heavily on the fact that the inverse system of an artinian Gorenstein ideal is generated by a single polynomial. In order to study unexpected systems of parameters of Cohen-Macaulay complexes that are not homology spheres, new techniques would have to be developed. In view of our results in~\cref{s:collapsible}, we ask the following.

\begin{question}\label{q:nonexistence}
    Let $\Delta$ be a $d$-dimensional contractible (or acyclic) complex. Can $\Delta$ have an unexpected sop?
\end{question}

It should be noted that a negative answer to~\cref{q:nonexistence} could have interesting consequences to the study of Lefschetz properties of monomial artinian reductions of squarefree monomial ideals. Every current proof that shows a given sop is unexpected can be understood from the fact that some multiplication map mimics the boundary map of a given simplicial complex $\Gamma$ defined in terms of $\Delta$, when applied to a specific set of linearly independent polynomials (see~\cite[Remark 4.3]{H2025B},~\cref{r:boundary,p:highermatrices}). A negative answer to~\cref{q:nonexistence} could then be seen as a consequence of a more general phenomenon relating multiplication and boundary maps. The first time such a connection was noticed was in~\cite{MNS2020}.

The notion of unexpectedness introduced earlier in this paper was designed to detect failure of WLP for monomial ideals due to surjectivity (as can be seen from~\cref{p:unexpectedfailure}). A natural question then is whether one could modify~\cref{d:unexpected} to be able to detect failure due to injectivity. Such a modified definition would probably involve notions of "expected" and "actual" dimensions, just as in the geometry setting~\cite{CHMN2018}. It is well known that monomial artinian reductions of Gorenstein squarefree monomial ideals can fail the WLP due to injectivity. Below we give an example where following the naive translation of the notion of unexpectedness, we can still say an unexpected sop exists.

\begin{example}[\emph{Unexpectedness beyond surjectivity}]\label{ex:injectivity}
    Let $I_\Delta = (x_1 x_2, x_3 x_4, x_5 x_6) \subset R = \K[x_1,\dots, x_6]$ where $\Delta$ is the boundary of an octahedron, $J = I_\Delta + (x_1^5, \dots, x_6^5)$ and $L = x_1 + \dots + x_6$. Then $A_\Delta(5) = \frac{R}{J}$ fails the WLP due to the map $\times L: A_\Delta(5)_{6} \to A_\Delta(5)_7$ not being injective: $\dim A_\Delta(5)_6 = 116 \leq 120 = \dim A_{\Delta}(5)_7$, and the kernel of the transpose is $5$-dimensional. A Macaulay2 computation gives us a basis $F_1, \dots, F_5$ for this kernel. Let $I_i$ denote the ideal such that $(I_i)^{-1} = (F_i)$, where $F_i$ is viewed as an element of $R$. We have the following information regarding the inverse systems of $F_1, \dots, F_5$:
    
    \begin{enumerate}
        \item $I_1$ has $6$ generators of degrees $1,1,2,2,3,4$,
        \item $I_2$ has $6$ generators of degrees $1,2,2,2,3,3$,
        \item $I_3$ has $6$ generators of degrees $1,1,2,2,3,4$,
        \item $I_4$ has $12$ generators, and 
        \item $I_5$ has $6$ generators of degrees $1,1,2,2,3,4$.
    \end{enumerate}
    
    In particular, out of the $5$ generators of the kernel basis, we conclude $4$ are the Macaulay dual generator of a complete intersection. Again a Macaulay2~\cite{M2} computation shows
    \begin{equation}\label{eq:sop}
        I_2 = I_\Delta + (\theta_1, \theta_2, \theta_3),        
    \end{equation} 
    where $\theta_1 = L$, $\theta_2 = x_5^3 + x_6^3$ and $\theta_3 = 4\,x_{3}^{3}+4\,x_{4}^{3}+6\,x_{3}^{2}x_{5}+6\,x_{4}^{2}x_{5}+9\,x_{3}x_{5}^{2}+9\,x_{4}x_{5}^{2}+6\,x_{3}^{2}x_{6}+6\,x_{4}^{2}x_{6}+9\,x_{3}x_{6}^{2}+9\,x_{4}x_{6}^{2}$.
    The sop in~\eqref{eq:sop} satisfies~\ref{i:1},~\ref{i:3} and~\ref{i:4} in~\cref{d:unexpected}. If instead of~\ref{i:5}, we consider the inequality 
    $$
        \dim \ker \times L^{T}: A_\Delta(5)_6 \to A_\Delta(5)_7 > \max(\dim A_\Delta(5)_7 - \dim A_\Delta(5)_6, 0),
    $$
    then the sop in~\eqref{eq:sop} can be seen as an unexpected sop of $\Delta$.     
\end{example}

\cref{ex:injectivity} not only shows the concept of an unexpected sop still makes sense when dealing with failure due to injectivity, but it once again hints at an interesting application of unexpected sops for monomial ideals. Determining, or even being able to generate (nontrivial) homogeneous polynomials $F$ such that $(F) = I^{-1}$, where $I$ is a complete intersection is a very hard problem in commutative algebra. To the best of our knowledge, the only two cases where the characterization for $F$ is known are: when $F$ is a monomial, in which case $F$ is always the Macaulay dual generator of a complete intersection, and when $F$ is a binomial. A characterization for the latter was recently obtained in~\cite{ADFMMSV2025,DM2025}.

Instead of directly trying to obtain a characterization of the Macaulay dual generators of complete intersections, we ask the following.

\begin{question}\label{q:unexpectedCI}
    Let $F_1, \dots, F_s$ be linearly independent homogeneous polynomials of degree $t$. When is there a polynomial $F$ that is a linear combination of $F_1, \dots, F_s$ such that $F$ is the Macaulay dual generator of a complete intersection?

    More specifically, let $I \subset R = \K[x_1, \dots, x_n]$ be a (squarefree) monomial complete intersection such that $x_i^{a_i} \not \in I$ for every $a_i > 0$ and every $i$, where $\K$ is a field of characteristic zero. Assume $A = \frac{R}{I + (x_1^{a_1},\dots, x_n^{a_n})}$ fails the WLP in degree $t - 1$, where $a_i > 2$ for every $i$. Let $K$ be the kernel of the transpose of the multiplication map where $A$ fails the WLP. Is there a polynomial $F \in K$ such that $F$ is the Macaulay dual generator of a complete intersection of the form 
    $$
        I + (\theta_1, \dots, \theta_r),
    $$
    where $\theta_1, \dots, \theta_r$ is a sop of $I$?
\end{question}

Even though the second part of~\cref{q:unexpectedCI} is probably too strong to be true, we do not know of a counter example.

Combinatorially, the condition that a squarefree monomial ideal $I_\Delta$ is a complete intersection can be stated by saying $I_\Delta$ is the Stanley-Reisner ideal of a join of boundary of simplices (possibly of different dimensions). Since in this case $\Delta$ is a homology sphere, both~\cite[Theorem 1.1]{H2025B} and~\cref{t:coloredunexpected} can be applied in order to generate Macaulay dual generators of complete intersections.
  
Finally, our results show that one can use information on the Rees algebras of special facet ideals in order to generate special systems of parameters for squarefree monomial ideals. In particular, the study of Rees-theoretic properties of edgewise subdivisions (and related constructions, such as the one introduced in~\cref{s:collapsible}) of arbitrary simplicial complexes becomes an interesting problem, that to the best of our knowledge has not been explored yet. This direction should be contrasted with the results in~\cite{CJKW2018}, where the authors apply topological properties of edgewise subdivision in order to understand free resolutions of the Stanley-Reisner ideals of complexes obtained after iterated edgewise subdivisions.

\begin{acknowledgement}
I would like to thank Alexandra Seceleanu, Brian Harbourne and my advisor Sara Faridi for insightful conversations that improved the exposition of this paper and Yirong Yang for pointing me to~\cite{KN2016}.
\end{acknowledgement}

\bibliographystyle{plain}
\bibliography{bibliography.bib}

\begin{bibdiv}
\begin{biblist}

\bib{AR2023}{article}{
      author={Adams, Ashleigh},
      author={Reiner, Victor},
       title={A colorful {H}ochster formula and universal parameters for face
  rings},
        date={2023},
        ISSN={1939-0807,1939-2346},
     journal={J. Commut. Algebra},
      volume={15},
      number={2},
       pages={151\ndash 176},
         url={https://doi.org/10.1216/jca.2023.15.151},
}

\bib{A2018}{article}{
      author={Adiprasito, Karim},
       title={Combinatorial {L}efschetz theorems beyond positivity},
        date={2018},
     journal={arXiv:1812.10454},
}

\bib{APP2021}{article}{
      author={Adiprasito, Karim},
      author={Papadakis, S.~Argyrios},
      author={Petrotou, Vasiliki},
       title={Anisotropy, biased pairings, and the {L}efschetz property for
  pseudomanifolds and cycles.},
        date={2021},
     journal={arXiv:2101.07245},
}

\bib{ADFMMSV2025}{article}{
      author={Altafi, Nasrin},
      author={Dinu, Rodica},
      author={Faridi, Sara},
      author={Shreedevi K.~Masuti, Rosa M. Mir\'o-Roig},
      author={Seceleanu, Alexandra},
      author={Villamizar, Nelly},
       title={Artinian {G}orenstein algebras with binomial {M}acaulay dual
  generator},
        date={2025},
     journal={arXiv:2502.18149},
}

\bib{BZ2011}{article}{
      author={Benedetti, Bruno},
      author={Ziegler, G\"unter~M.},
       title={On locally constructible spheres and balls},
        date={2011},
        ISSN={0001-5962,1871-2509},
     journal={Acta Math.},
      volume={206},
      number={2},
       pages={205\ndash 243},
         url={https://doi.org/10.1007/s11511-011-0062-2},
      review={\MR{2810852}},
}

\bib{B1995}{incollection}{
      author={Bj\"orner, A.},
       title={Topological methods},
        date={1995},
   booktitle={Handbook of combinatorics, {V}ol.\ 1,\ 2},
   publisher={Elsevier Sci. B. V., Amsterdam},
       pages={1819\ndash 1872},
      review={\MR{1373690}},
}

\bib{BW2009}{article}{
      author={Brenti, Francesco},
      author={Welker, Volkmar},
       title={The {V}eronese construction for formal power series and graded
  algebras},
        date={2009},
        ISSN={0196-8858,1090-2074},
     journal={Adv. in Appl. Math.},
      volume={42},
      number={4},
       pages={545\ndash 556},
         url={https://doi.org/10.1016/j.aam.2009.01.001},
      review={\MR{2511015}},
}

\bib{BR2005}{article}{
      author={Brun, Morten},
      author={R\"omer, Tim},
       title={Subdivisions of toric complexes},
        date={2005},
        ISSN={0925-9899,1572-9192},
     journal={J. Algebraic Combin.},
      volume={21},
      number={4},
       pages={423\ndash 448},
         url={https://doi.org/10.1007/s10801-005-3020-2},
      review={\MR{2153934}},
}

\bib{CJKW2018}{article}{
      author={Conca, Aldo},
      author={Juhnke-Kubitzke, Martina},
      author={Welker, Volkmar},
       title={Asymptotic syzygies of {S}tanley-{R}eisner rings of iterated
  subdivisions},
        date={2018},
        ISSN={0002-9947,1088-6850},
     journal={Trans. Amer. Math. Soc.},
      volume={370},
      number={3},
       pages={1661\ndash 1691},
         url={https://doi.org/10.1090/tran/7149},
      review={\MR{3739188}},
}

\bib{CHMN2018}{article}{
      author={Cook, D., II},
      author={Harbourne, B.},
      author={Migliore, J.},
      author={Nagel, U.},
       title={Line arrangements and configurations of points with an unexpected
  geometric property},
        date={2018},
        ISSN={0010-437X,1570-5846},
     journal={Compos. Math.},
      volume={154},
      number={10},
       pages={2150\ndash 2194},
         url={https://doi.org/10.1112/s0010437x18007376},
      review={\MR{3867298}},
}

\bib{CJKMN2018}{article}{
      author={Cook, David, II},
      author={Juhnke-Kubitzke, Martina},
      author={Murai, Satoshi},
      author={Nevo, Eran},
       title={Lefschetz properties of balanced 3-polytopes},
        date={2018},
        ISSN={0035-7596,1945-3795},
     journal={Rocky Mountain J. Math.},
      volume={48},
      number={3},
       pages={769\ndash 790},
         url={https://doi.org/10.1216/RMJ-2018-48-3-769},
      review={\MR{3835571}},
}

\bib{DN2024}{article}{
      author={Dao, Hailong},
      author={Nair, Ritika},
       title={On the {L}efschetz property for quotients by monomial ideals
  containing squares of variables},
        date={2024},
        ISSN={0092-7872,1532-4125},
     journal={Comm. Algebra},
      volume={52},
      number={3},
       pages={1260\ndash 1270},
         url={https://doi.org/10.1080/00927872.2023.2260012},
      review={\MR{4709588}},
}

\bib{DCEP1982}{book}{
      author={De~Concini, Corrado},
      author={Eisenbud, David},
      author={Procesi, Claudio},
       title={Hodge algebras},
      series={Ast\'erisque},
   publisher={Soci\'et\'e{} Math\'ematique de France, Paris},
        date={1982},
      volume={91},
        note={With a French summary},
}

\bib{DE2009}{article}{
      author={Dochtermann, Anton},
      author={Engstr\"om, Alexander},
       title={Algebraic properties of edge ideals via combinatorial topology},
        date={2009},
        ISSN={1077-8926},
     journal={Electron. J. Combin.},
      volume={16},
      number={2},
       pages={Research Paper 2, 24},
         url={https://doi.org/10.37236/68},
      review={\MR{2515765}},
}

\bib{EG2000}{incollection}{
      author={Edelsbrunner, H.},
      author={Grayson, D.~R.},
       title={Edgewise subdivision of a simplex},
        date={2000},
      volume={24},
       pages={707\ndash 719},
         url={https://doi.org/10.1145/304893.304897},
        note={ACM Symposium on Computational Geometry (Miami, FL, 1999)},
      review={\MR{1799608}},
}

\bib{EI1995}{article}{
      author={Emsalem, J.},
      author={Iarrobino, A.},
       title={Inverse system of a symbolic power. {I}},
        date={1995},
        ISSN={0021-8693,1090-266X},
     journal={J. Algebra},
      volume={174},
      number={3},
       pages={1080\ndash 1090},
         url={https://doi.org/10.1006/jabr.1995.1168},
      review={\MR{1337186}},
}

\bib{FHM2020}{article}{
      author={Fouli, Louiza},
      author={H\`a, Huy~T\`ai},
      author={Morey, Susan},
       title={Initially regular sequences and depths of ideals},
        date={2020},
        ISSN={0021-8693,1090-266X},
     journal={J. Algebra},
      volume={559},
       pages={33\ndash 57},
         url={https://doi.org/10.1016/j.jalgebra.2020.03.027},
      review={\MR{4093701}},
}

\bib{GS1984}{article}{
      author={Garsia, A.~M.},
      author={Stanton, D.},
       title={Group actions of {S}tanley-{R}eisner rings and invariants of
  permutation groups},
        date={1984},
        ISSN={0001-8708},
     journal={Adv. in Math.},
      volume={51},
      number={2},
       pages={107\ndash 201},
         url={https://doi.org/10.1016/0001-8708(84)90005-7},
}

\bib{DM2025}{article}{
      author={Gennaro, Roberta~Di},
      author={Mir\'o-Roig, Rosa~M.},
       title={Complete intersection algebras with binomial {M}acaulay dual
  generator},
        date={2025},
     journal={arXiv:2503.13920},
}

\bib{G2001}{book}{
      author={Godsil, Chris},
      author={Royle, Gordon},
       title={Algebraic graph theory},
      series={Graduate Texts in Mathematics},
   publisher={Springer-Verlag, New York},
        date={2001},
      volume={207},
        ISBN={0-387-95241-1; 0-387-95220-9},
         url={https://doi.org/10.1007/978-1-4613-0163-9},
      review={\MR{1829620}},
}

\bib{M2}{misc}{
      author={Grayson, Daniel~R.},
      author={Stillman, Michael~E.},
       title={Macaulay2, a software system for research in algebraic geometry},
         how={Available at \url{http://www2.macaulay2.com}},
}

\bib{HKZ2022}{article}{
      author={Harbourne, Brian},
      author={Kettinger, Jake},
      author={Zimmitti, Frank},
       title={Extreme values of the resurgence for homogeneous ideals in
  polynomial rings},
        date={2022},
        ISSN={0022-4049,1873-1376},
     journal={J. Pure Appl. Algebra},
      volume={226},
      number={2},
       pages={Paper No. 106811, 16},
         url={https://doi.org/10.1016/j.jpaa.2021.106811},
      review={\MR{4276485}},
}

\bib{HMN2024}{incollection}{
      author={Harbourne, Brian},
      author={Migliore, Juan},
      author={Nagel, Uwe},
       title={Unexpected hypersurfaces and their consequences: a survey},
        date={[2024] \copyright 2024},
   booktitle={Lefschetz properties---current and new directions},
      series={Springer INdAM Ser.},
      volume={59},
   publisher={Springer, Singapore},
       pages={29\ndash 58},
         url={https://doi.org/10.1007/978-981-97-3886-1_2},
      review={\MR{4807659}},
}

\bib{lefschetzbook}{book}{
      author={Harima, Tadahito},
      author={Maeno, Toshiaki},
      author={Morita, Hideaki},
      author={Numata, Yasuhide},
      author={Wachi, Akihito},
      author={Watanabe, Junzo},
       title={The {L}efschetz properties},
      series={Lecture Notes in Mathematics},
   publisher={Springer, Heidelberg},
        date={2013},
      volume={2080},
        ISBN={978-3-642-38205-5; 978-3-642-38206-2},
         url={https://doi.org/10.1007/978-3-642-38206-2},
}

\bib{H2005}{article}{
      author={Hausel, Tam\'as},
       title={Quaternionic geometry of matroids},
        date={2005},
        ISSN={1895-1074,1644-3616},
     journal={Cent. Eur. J. Math.},
      volume={3},
      number={1},
       pages={26\ndash 38},
         url={https://doi.org/10.2478/BF02475653},
}

\bib{HM2021}{article}{
      author={Herzog, J\"urgen},
      author={Moradi, Somayeh},
       title={Systems of parameters and the {C}ohen-{M}acaulay property},
        date={2021},
        ISSN={0925-9899,1572-9192},
     journal={J. Algebraic Combin.},
      volume={54},
      number={4},
       pages={1261\ndash 1277},
         url={https://doi.org/10.1007/s10801-021-01046-6},
}

\bib{H2024}{article}{
      author={Holleben, Thiago},
       title={The weak {L}efschetz property and mixed multiplicities of
  monomial ideals},
        date={2024},
        ISSN={0925-9899,1572-9192},
     journal={J. Algebraic Combin.},
      volume={60},
      number={1},
       pages={295\ndash 317},
         url={https://doi.org/10.1007/s10801-024-01337-8},
      review={\MR{4781163}},
}

\bib{H2025B}{article}{
      author={Holleben, Thiago},
       title={Coinvariant stresses, {L}efschetz properties and random
  complexes},
        date={2025},
     journal={arXiv:2501.12108},
}

\bib{H2025}{article}{
      author={Holleben, Thiago},
       title={Lefschetz properties of squarefree monomial ideals via {R}ees
  algebras},
        date={2025},
        ISSN={0021-8693,1090-266X},
     journal={J. Algebra},
      volume={668},
       pages={572\ndash 599},
         url={https://doi.org/10.1016/j.jalgebra.2024.12.030},
      review={\MR{4864830}},
}

\bib{HN2025}{article}{
      author={Holleben, Thiago},
      author={Nicklasson, Lisa},
       title={Roller {C}oaster {G}orenstein algebras and {K}oszul {G}orenstein
  algebras failing the weak {L}efschetz property},
        date={2025},
     journal={arXiv:2502.00155},
}

\bib{J2002}{article}{
      author={Joswig, Michael},
       title={Projectivities in simplicial complexes and colorings of simple
  polytopes},
        date={2002},
        ISSN={0025-5874,1432-1823},
     journal={Math. Z.},
      volume={240},
      number={2},
       pages={243\ndash 259},
         url={https://doi.org/10.1007/s002090100381},
      review={\MR{1900311}},
}

\bib{JKM2018}{article}{
      author={Juhnke-Kubitzke, Martina},
      author={Murai, Satoshi},
       title={Balanced generalized lower bound inequality for simplicial
  polytopes},
        date={2018},
        ISSN={1022-1824,1420-9020},
     journal={Selecta Math. (N.S.)},
      volume={24},
      number={2},
       pages={1677\ndash 1689},
         url={https://doi.org/10.1007/s00029-017-0363-1},
      review={\MR{3782432}},
}

\bib{KNN2016}{article}{
      author={Kalai, Gil},
      author={Nevo, Eran},
      author={Novik, Isabella},
       title={Bipartite rigidity},
        date={2016},
        ISSN={0002-9947,1088-6850},
     journal={Trans. Amer. Math. Soc.},
      volume={368},
      number={8},
       pages={5515\ndash 5545},
         url={https://doi.org/10.1090/tran/6512},
      review={\MR{3458389}},
}

\bib{KN2016}{article}{
      author={Klee, Steven},
      author={Novik, Isabella},
       title={Lower bound theorems and a generalized lower bound conjecture for
  balanced simplicial complexes},
        date={2016},
        ISSN={0025-5793,2041-7942},
     journal={Mathematika},
      volume={62},
      number={2},
       pages={441\ndash 477},
         url={https://doi.org/10.1112/S0025579315000297},
      review={\MR{3521335}},
}

\bib{KW2012}{article}{
      author={Kubitzke, Martina},
      author={Welker, Volkmar},
       title={Enumerative {$g$}-theorems for the {V}eronese construction for
  formal power series and graded algebras},
        date={2012},
        ISSN={0196-8858,1090-2074},
     journal={Adv. in Appl. Math.},
      volume={49},
      number={3-5},
       pages={307\ndash 325},
         url={https://doi.org/10.1016/j.aam.2012.08.002},
      review={\MR{3017962}},
}

\bib{macaulaybookinversesystems}{book}{
      author={Macaulay, F.~S.},
       title={The algebraic theory of modular systems},
      series={Cambridge Mathematical Library},
   publisher={Cambridge University Press, Cambridge},
        date={1994},
        ISBN={0-521-45562-6},
        note={Revised reprint of the 1916 original, With an introduction by
  Paul Roberts},
}

\bib{MMV2012}{article}{
      author={Mart\'{\i}nez-Bernal, Jos\'{e}},
      author={Morey, Susan},
      author={Villarreal, {Rafael H.}},
       title={Associated primes of powers of edge ideals},
        date={2012},
        ISSN={0010-0757},
     journal={Collect. Math.},
      volume={63},
      number={3},
       pages={361\ndash 374},
  url={https://doi-org.ezproxy.library.dal.ca/10.1007/s13348-011-0045-9},
}

\bib{MNS2020}{article}{
      author={Migliore, Juan},
      author={Nagel, Uwe},
      author={Schenck, Hal},
       title={The weak {L}efschetz property for quotients by quadratic
  monomials},
        date={2020},
        ISSN={0025-5521,1903-1807},
     journal={Math. Scand.},
      volume={126},
      number={1},
       pages={41\ndash 60},
         url={https://doi.org/10.7146/math.scand.a-116681},
      review={\MR{4087572}},
}

\bib{MMN2011}{article}{
      author={Migliore, Juan~C.},
      author={Mir\'o-Roig, Rosa~M.},
      author={Nagel, Uwe},
       title={Monomial ideals, almost complete intersections and the weak
  {L}efschetz property},
        date={2011},
        ISSN={0002-9947,1088-6850},
     journal={Trans. Amer. Math. Soc.},
      volume={363},
      number={1},
       pages={229\ndash 257},
         url={https://doi.org/10.1090/S0002-9947-2010-05127-X},
}

\bib{MV2015}{inproceedings}{
      author={Mirzakhani, Maryam},
      author={Vondr\'ak, Jan},
       title={Sperner's colorings, hypergraph labeling problems and fair
  division},
        date={2015},
   booktitle={Proceedings of the {T}wenty-{S}ixth {A}nnual {ACM}-{SIAM}
  {S}ymposium on {D}iscrete {A}lgorithms},
   publisher={SIAM, Philadelphia, PA},
       pages={873\ndash 886},
         url={https://doi.org/10.1137/1.9781611973730.60},
      review={\MR{3451084}},
}

\bib{MV2012}{incollection}{
      author={Morey, Susan},
      author={Villarreal, Rafael~H.},
       title={Edge ideals: algebraic and combinatorial properties},
        date={2012},
   booktitle={Progress in commutative algebra 1},
   publisher={de Gruyter, Berlin},
       pages={85\ndash 126},
      review={\MR{2932582}},
}

\bib{M2011}{article}{
      author={Murai, Satoshi},
       title={Spheres arising from multicomplexes},
        date={2011},
        ISSN={0097-3165,1096-0899},
     journal={J. Combin. Theory Ser. A},
      volume={118},
      number={8},
       pages={2167\ndash 2184},
         url={https://doi.org/10.1016/j.jcta.2011.04.015},
      review={\MR{2834171}},
}

\bib{MN2013}{article}{
      author={Murai, Satoshi},
      author={Nevo, Eran},
       title={On the generalized lower bound conjecture for polytopes and
  spheres},
        date={2013},
        ISSN={0001-5962,1871-2509},
     journal={Acta Math.},
      volume={210},
      number={1},
       pages={185\ndash 202},
         url={https://doi.org/10.1007/s11511-013-0093-y},
}

\bib{O2024}{article}{
      author={Oba, Ryoshun},
       title={Multigraded strong {L}efschetz property for balanced simplicial
  complexes},
        date={2024},
     journal={arXiv:2408.17110},
}

\bib{PP2020}{article}{
      author={Papadakis, S.~Argyrios},
      author={Petrotou, Vasiliki},
       title={The characteristic $2$ anisotropicity of simplicial spheres.},
        date={2020},
     journal={arXiv:2012.09815v1},
}

\bib{R1976}{article}{
      author={Reisner, Gerald~Allen},
       title={Cohen-{M}acaulay quotients of polynomial rings},
        date={1976},
        ISSN={0001-8708},
     journal={Advances in Math.},
      volume={21},
      number={1},
       pages={30\ndash 49},
         url={https://doi.org/10.1016/0001-8708(76)90114-6},
      review={\MR{407036}},
}

\bib{S1990}{article}{
      author={Smith, Dean~E.},
       title={On the {C}ohen-{M}acaulay property in commutative algebra and
  simplicial topology},
        date={1990},
        ISSN={0030-8730,1945-5844},
     journal={Pacific J. Math.},
      volume={141},
      number={1},
       pages={165\ndash 196},
         url={http://projecteuclid.org/euclid.pjm/1102646778},
}

\bib{S1975}{article}{
      author={Stanley, Richard~P.},
       title={The upper bound conjecture and {C}ohen-{M}acaulay rings},
        date={1975},
        ISSN={0022-2526,1467-9590},
     journal={Studies in Appl. Math.},
      volume={54},
      number={2},
       pages={135\ndash 142},
         url={https://doi.org/10.1002/sapm1975542135},
      review={\MR{458437}},
}

\bib{S1979}{article}{
      author={Stanley, Richard~P.},
       title={Balanced {C}ohen-{M}acaulay complexes},
        date={1979},
        ISSN={0002-9947,1088-6850},
     journal={Trans. Amer. Math. Soc.},
      volume={249},
      number={1},
       pages={139\ndash 157},
         url={https://doi.org/10.2307/1998915},
      review={\MR{526314}},
}

\bib{S80}{article}{
      author={Stanley, Richard~P.},
       title={Weyl groups, the hard {L}efschetz theorem, and the {S}perner
  property},
        date={1980},
        ISSN={0196-5212},
     journal={SIAM J. Algebraic Discrete Methods},
      volume={1},
      number={2},
       pages={168\ndash 184},
}

\bib{S1996B}{book}{
      author={Stanley, Richard~P.},
       title={Combinatorics and commutative algebra},
     edition={Second},
      series={Progress in Mathematics},
   publisher={Birkh\"auser Boston Inc. Boston MA},
        date={1996},
      volume={41},
        ISBN={0-8176-3836-9},
}

\bib{T2001}{article}{
      author={Trung, Ng\^{o}~Vi\^{e}t},
       title={Positivity of mixed multiplicities},
        date={2001},
        ISSN={0025-5831,1432-1807},
     journal={Math. Ann.},
      volume={319},
      number={1},
       pages={33\ndash 63},
         url={https://doi.org/10.1007/PL00004429},
}

\bib{VTZ2010}{article}{
      author={Van~Tuyl, Adam},
      author={Zanello, Fabrizio},
       title={Simplicial complexes and {M}acaulay's inverse systems},
        date={2010},
        ISSN={0025-5874,1432-1823},
     journal={Math. Z.},
      volume={265},
      number={1},
       pages={151\ndash 160},
         url={https://doi.org/10.1007/s00209-009-0507-x},
      review={\MR{2606954}},
}

\bib{Z1964}{article}{
      author={Zeeman, E.~C.},
       title={On the dunce hat},
        date={1964},
        ISSN={0040-9383},
     journal={Topology},
      volume={2},
       pages={341\ndash 358},
         url={https://doi.org/10.1016/0040-9383(63)90014-4},
      review={\MR{156351}},
}

\end{biblist}
\end{bibdiv}

\end{document}